\documentclass[12pt,reqno,english]{amsart}

\usepackage[left=3cm,top=3cm,right=3cm,bottom=3cm]{geometry}

\usepackage{amsfonts,amsmath,latexsym,verbatim,amscd,mathrsfs,color,array}
\usepackage{amsmath,amssymb,amsthm,amsfonts,graphicx,color}
\usepackage{amssymb}
\usepackage{pdfsync}
\usepackage{epstopdf}
\usepackage[english]{babel}
\usepackage[latin1]{inputenc} %%% este comando sierve para que las tildes se pongan como es usual%%%
\usepackage{amsmath}
\usepackage{amssymb}
\usepackage{amsfonts}
\usepackage{amsthm}
\usepackage{latexsym}
\usepackage{bbold}
\usepackage[active]{srcltx}
\usepackage{hyperref}
\numberwithin{equation}{section}
\usepackage{color}      %Need the color package
\usepackage{epsfig}
\usepackage{cmap}
\usepackage{cite}

\newcommand{\R}{\mathbb{R}}

\newcommand{\LL}{\mathcal{L}}
\newcommand{\OO}{\mathcal{O}}

\newcommand{\pp}{\partial}

\newcommand{\A}{\alpha}
\newcommand{\D}{\delta}

\newcommand{\ep}{\varepsilon}

\newcommand{\rdiv}{{\rm div}}
\newcommand{\rdist}{{\rm dist}}
\newcommand{\bfs}{{\bf s}}

\newcommand{\bft}{{\bf t}}
\newcommand{\zzeta}{{\boldsymbol  \zeta}}

\DeclareMathOperator{\uni}{\cup}
\newtheorem{lemma}{Lemma}[section]
\newtheorem{prop}{Proposition}[section]
\newtheorem{theo}{Theorem}
\newtheorem{corol}{Corollary}[section]
\title[\resizebox{5.6in}{!}{Boundary concentration phenomena for  the higher-dimensional Keller-Segel system}]{Boundary concentration phenomena for  the higher-dimensional Keller-Segel system}

\author[O. Agudelo]{Oscar Agudelo}
\address{\noindent O. Agudelo - NTIS Department of Mathematics,
Z\'{a}pado\v{c}esk\'{a} Univerzita v Plzni}
\email{oiagudel@ntis.zcu.cz}

\author[A. Pistoia]{Angela Pistoia}
\address{\noindent A. Pistoia - Dipartimento di Scienze Base e Applicate  La Sapienza Universit\'{a} di Roma.}
\email{angela.pistoia@uniroma1.it}

\begin{document}
\subjclass[2010]{35J60 (primary), and 35B33, 35J20 (secondary)}
\keywords{Keller-Segel system,  boundary and interior concentration layer}

\begin{abstract}
We study the existence of  steady states to the  Keller-Segel
 system   with linear chemotactical sensitivity function on a smooth bounded domain in $\R^N,$ $N\ge3,$ having rotational symmetry. 
 We find three different types of   chemoattractant concentration which  concentrate along   suitable $(N-2)-$dimensional minimal submanifolds of the boundary.
The corresponding density of the cellular slime molds 
 exhibit in the limit one or more Dirac measures supported on those boundary submanifolds.
\end{abstract}
\maketitle
\tableofcontents

\section{Introduction}

In 1970 Keller and Segel \cite{KELLERSEGEL} presented a system of two strongly coupled parabolic PDE's to describe the aggregation of cellular slime molds like Dictyostelium Discoidem. Assuming the whole process to take place on a suitable bounded region $D$ in $\R^N,$ $N\ge1$,
with no flux across the boundary, the myxoamoebae density of the cellular slime molds $w(t,x)$ and the chemoattractant concentration $v(t,x)$ at time $t$ in a point $x$ in $D$ satisfy the system
\begin{equation}\label{ks}
\left\{\begin{aligned}
&\partial_t w=\nabla\left[\mu(w,v)\nabla w-\chi(w,v)\nabla v\right]\ \hbox{in}\ \R\times D\\
&\partial_t v=\gamma_0\Delta v+k(w,v)\ \hbox{in}\ \R\times D\\
&\partial_\nu w=\partial_\nu v=0\ \hbox{in}\ \R\times\partial D\\
&w(0,x)=w_0(x),\ v(0,x)=v_0(x)\ \hbox{in}\  D,
\end{aligned}
\right.
\end{equation}
where
\begin{itemize}
\item $\mu(w,v)$ is the random motility coefficient
\item $\chi(w,v)=\chi_0\mu(w,v)w \nabla \Phi(v)$ is the total chemotaxic flux, where $\chi_0$ 
is a   constant and $\Phi$ is a smooth increasing function called {\it chemotactic sensitivity function}
\item $\gamma_0>0$ is a constant diffusion coefficient
\item $k(w,v)$ models the reaction, which   commonly is $k(w,v)=\gamma_0(\alpha w-\beta v) $ for some constants $\alpha$ and $\beta$.  
\item $\nu$ is the unit inner normal derivative at the boundary.
\end{itemize}

This model has attracted the attention of many mathematicians, since it has led to a
variety of stimulating problems.
Many contributions have been made towards understanding analytical aspects of system \eqref{ks}. We refer the reader to \cite{BILER,BRENNERCONSTANTINKADANOFFSCHENKELVENKATARAMANI,CHILDRESS,CORRIASPERTHAMEZAAQ,DOLBEAULTPERTHAME,GUERRAPELETIER,HERREROVELAZQUEZ1,HERREROVELAZQUEZ2,HERREROVELAZQUEZ3,HORSTMANN,JAGERWAND} and references therein. In particular, we quote  the recent  survey \cite{BBTW} which  focuses on the origi\-nal model and some of its developments and is devoted to the qualitative analysis of analytic problems, such as the existence of solutions, blow-up and asymptotic behavior. \\

The understanding of the global dynamics of the system is strongly related to
the existence of steady states, namely solutions to the system
\begin{equation}\label{kss} 
\left\{\begin{aligned}
& \nabla\left[\mu(w,v)\nabla w-\chi(w,v)\nabla v\right]=0\ \hbox{in}\   D\\
& \gamma_0\Delta v+k(w,v)=0\ \hbox{in}\   D\\
&\partial_\nu w=\partial_\nu v=0\ \hbox{in}\  \partial D.\\
 \end{aligned}
\right.
\end{equation}

The present paper deals with the existence of stationary solutions to the problem \eqref{ks} when the chemotactical sensitivity function $\Phi$ is linear, i.e.
$\Phi(v)=v$,.
In such a case the study of \eqref{kss} reduces to a single equation. Indeed, the function $w(x)=\lambda e^{\chi_0 v}$ solves the first equation and the second equation reduces to
\begin{equation}\label{lin-sen}
\Delta v+\alpha \lambda e^{\chi_0 v}-\beta v=0\ \hbox{in}\  D,\quad \partial_\nu v=0\ \hbox{on}\ \partial D.\end{equation}

Here we study problem \eqref{lin-sen}, which without loss of generality can  be rewritten as
\begin{equation}\label{p}
\Delta v+  \ep^2 e^{  v}-  v=0\ \hbox{in}\  D,\quad \partial_\nu v=0\ \hbox{on}\ \partial D.\end{equation}
when $ D$ is  a bounded domain in $\R^N,$ $N\ge1$, $\nu$ is the inner unit normal vector to $\pp \Omega$
and $\ep$ is a small  parameter.\\

The one dimensional version of equation \eqref{p}, was first treated by Schaaf in \cite{SCHAAF}.  In the two-dimensional case,   Wang and Wei \cite{WANGWEI} and Senba and Suzuki \cite{SENBASUZUKI} proved the existence of a non-costant solution for $\ep\le \ep_0$ for some $\ep_0$ and $\ep$ possibly different from certain values depending on the domain. 
Successively, del Pino and Wei in \cite{DELPINOWEI} constructed   solutions to \eqref{p} which concentrate (as $\ep$ goes to zero) at $\kappa$ different points $\xi_1,\dots,\xi_\kappa$ on the boundary of $ D$ and $\ell$ different points $\xi_{\kappa+1},\dots,\xi_{\kappa+\ell}$ inside the domain $ D.$ In particular
far away from those points the leading behavior of $v_\ep$ is given by
\begin{equation}\label{kk}
v_\ep(x) \to  \sum\limits_{i=1}^\kappa {1\over 2} G(x,\xi_i)  + \sum\limits_{i=1}^\ell  G(x, \xi_{i})
\end{equation}
where $G(\cdot,\xi)$ is the Green function for the problem
\begin{equation}
\left\{
\begin{aligned}
&-\Delta G + G = 8\pi \delta_\xi, \  \mbox{in}\   D,\\
& \frac{\partial G}{\partial\nu}=0\ \mbox{on}\ \partial D.
\end{aligned}
\right.
\end{equation}
 Here $\D_   {\zzeta}$ represents the   Dirac's mass concentrated at the point $\zeta$.
The corresponding solution $u_\ep(x)=\ep^2e^{v_\ep}$ of the first equation in \eqref{kss} exhibits, in the limit, $\kappa$ Dirac measures on the boundary of the domain and $\ell$ Dirac measures inside the domain with respectively weights $4\pi$ and $8\pi,$ namely
$$u_\ep\rightharpoonup \sum\limits_{i=1}^\kappa 4\pi\delta_{\xi_i}+\sum\limits_{i=1}^\ell 8\pi\delta_{\xi_{\kappa+i}}$$

Recently, Del Pino, Pistoia and Vaira in \cite{dpv} built a solution to problem \eqref{p} which concentrate along the whole boundary. In particular, far away from the boundary the leading behavior of $v_\ep$ is given by
$${1\over|\ln\ep|}v_\ep(x)\to\mathcal G(x)$$
where  $ \mathcal G$ is the unique solution of the problem
 \begin{equation}\label{G}
- \Delta  \mathcal G +  \mathcal G  =0 \hbox{ in }  D, \quad    \mathcal G = 1 \hbox{ on } \partial D.
\end{equation}

The corresponding solution $u_\ep=\ep^2e^{v_\ep}$ of the first equation in \eqref{kss} exhibits in the limit a Dirac measures supported on the boundary with a suitable weight, namely
$${1\over|\ln\ep|}u_\ep \rightharpoonup - \partial_\nu  \mathcal G\ \delta_{\partial D} .$$
($\partial_\nu  \mathcal G <0$, because of the maximum principle and Hopf's Lemma) .

As far as we know the only results dealing with higher-dimensional cases concerns the case when $D$ is a ball.  
Biler in \cite{BILER} established the existence of  a   strictly decreasing radial solutions, while Pistoia and Vaira in \cite{pv} found a second radial solution which is increasing and concentrates along the whole boundary of the domain as $\ep$ approaches zero. \\

Clearly, a natural  question arises.\\
 {\em Do there exist solutions to problem \eqref{p} in higher dimensional domains?
 In particular, is it possible to find solutions to problem \eqref{p} which concentrates on 
suitable submani\-folds of the boundary as the parameter $\ep$ approaches zero?}\\

In the present paper we give a positive answer when the domain has a rotational symmetry. 
Let $n=1,2,$ be fixed.   Let $\Omega$ be a smooth open bounded domain in $\R^2 $ such that
$$\overline {\Omega} \subset \{\left( x_{1},x_{2},x^{\prime
}\right) \in \mathbb{R}^{n}\times \mathbb{R}^{2-n}:x_{i}>0,\ i=1
,n\}.
$$
Let $M=\sum\limits_{i=1}^n M_i,$ $M_{i}\geq 2,$ and set
\begin{equation}\label{D}
{D}:=\{(y_{1},y_{n},x^{\prime })\in \mathbb{R}^{M_{1}}\times \mathbb{R}^{M_{n}}\times \mathbb{R}^{2-n}:\left( \left\vert
y_{1}\right\vert ,\left\vert y_{n}\right\vert ,x^{\prime }\right)
\in \Omega\}.
\end{equation}%
Then $ {D}$ is a smooth bounded domain in $\mathbb{R}^{N}$, $N:=M+2-n.$

The solutions we are looking for are $\Gamma-$invariant for the action of
the group $\Gamma:=\mathcal{O}(M_1)\times\mathcal{O}(M_n)$ on $\R^N$ given by
\begin{equation*}
(g_{1},g_{n})(y_{1},y_{n},x^{\prime }):=(g_{1}y_{1}
,g_{n}y_{n},x^{\prime }).
\end{equation*}%
Here $\mathcal{O}(M_i)$ denotes the group of linear isometries of $\R^{M_i}.$

A simple calculation shows that a function $v$ of the form 
\begin{equation}\label{simmetria}v(y_{1}
,y_{n},x^{\prime })=u\left( \left\vert y_{1}\right\vert ,\left\vert
y_{n}\right\vert ,x^{\prime }\right) \end{equation}
 solves problem \eqref{p} if and only
if $u$ solves the problem
\begin{equation*}
-\Delta u+\sum_{i=1}^{n}\frac{M_{i}-1}{x_{i}}\frac{\partial u}{\partial
x_{i}}+u=\ep^2 e^u\quad \text{in}\ \Omega,\qquad \partial_\nu u=0\quad \text{on}\
\partial \Omega,
\end{equation*}%
which can be rewritten as%
\begin{equation*}
-\text{div}(a(x)\nabla u)+a(x)u= \ep ^2 a(x)e^{u}\quad \text{in}\ \Omega%
 ,\qquad u=0\quad \text{on}\ \partial \Omega,
\end{equation*}%
where%
\begin{equation}\label{a}
a(x_{1},x_{n}):=x_{1}^{M_{1}-1}\cdot x_{n}^{M_{n}-1}.
\end{equation}
Thus, we are led to study the more general anisotropic equation
 
\begin{equation}\label{emdenfowlerhighdim}
-\rdiv(a(x)\nabla u) +a(x)u = \ep^2 a(x)e^{u}  \quad \hbox{in} \quad \Omega ,\quad \quad \frac{\pp u}{\pp \nu}=0 \quad \hbox{on} \quad \pp \Omega,
\end{equation}
where $\Omega\subset \R^2$ is a smooth bounded domain, $a: \overline{\Omega}\to \R$ is a strictly positive smooth function and $\ep>0$ is a small parameter. Here $\nu$ stands for the inner unit normal vector to $\pp \Omega$.

 Our goal is to construct solutions to problem \eqref{emdenfowlerhighdim} which concentrate at points $\zeta_1,\dots,\zeta_m$ of the boundary of $\Omega$ as $\epsilon$ goes to $0.$ They
correspond via \eqref{simmetria} to $\Gamma-$invariant solutions to  problem \eqref{p} with layers  which concentrate along  the  $\Gamma-$orbit $\Xi(\zeta_i)$ of $\zeta_i,$ for $i=1,\dots,m$
as  $\epsilon$ approaches zero. Here
\begin{equation}\label{orbit}
\Xi(\zeta_i):=\{(y_{1},y_{n},x^{\prime })\in \partial D:\left( \left\vert
y_{1}\right\vert ,\left\vert y_{n}\right\vert ,x^{\prime }\right)=\zeta_i
\in \partial\Omega\}
\end{equation}
is a $(N-2)-$dimensional  minimal submanifold  of the boundary of $D $
  diffeomorphic to $ \mathbb{S}^{M_1-1}\times\mathbb{S}^{M_n-1}  $
 where $\mathbb{S}^{M_i-1}$ is the unit sphere in $\R^{M_i} . $
\\\\
In order to state our main result, we need to introduce some tools.

The basic cell in our construction are the so-called {\em standard bubbles}
\begin{equation}\label{bubble}
U_{\mu,\zeta}(x):=\ln \left(8\mu^2\over(\mu^2+|x-\zeta|^2)^2\right),\ x,\zeta\in\R^2,\ \mu>0,
\end{equation}
which solve the Liouville equation
\begin{equation}\label{liu}
\Delta U+e^U=0\ \hbox{in}\ \R^2.
\end{equation}
To get a good approximation, we need to project the bubbles in order to fit the Neumann boundary condition with the linear
 differential operator
\begin{equation}\label{linearpartanisotropiceqn}
\LL_a u:= \Delta u+   \nabla (\ln a) \cdot \nabla u-u,
\end{equation}  
namely  
$$\LL_a PU_{\mu,\zeta}=\LL_a  U_{\mu,\zeta}\ \hbox{in}\ \Omega,\ \partial_\nu PU_{\mu,\zeta}=0\ \hbox{on}\ \pp\Omega.$$
To  compute the error given by the projected bubble, we need to perform
   a careful analysis of the regularity and asymptotic behavior of the Green's function $G_a(\cdot,\zeta)$ associate with 
 $\LL_a$ with Neumann boundary condition, namely  
\begin{equation} \label{greensfunctioneqn0}\left\{\begin{aligned}
&\LL_a \,G_a(\cdot,\zeta) + 8\pi\D_   {\zzeta}=0\quad \quad  \hbox{in}\quad \Omega,\\
&
\frac{\pp G_a(\cdot,\zeta)}{\pp \nu}=0\quad \quad  \hbox{on} \quad \pp \Omega,
\end{aligned}\right.\end{equation}
for every $\zeta\in \overline{\Omega}$.
The regular part of $G_a(\cdot,\zeta)$ is defined for $x\in \Omega$ as
\begin{equation}\label{regularpartGreen1}
H_a(x,\zeta):= 
\left\{
\begin{array}{ccc}
G_a(x,\zeta) + 4 \ln \left(|x-\zeta|\right),&\quad 
\zeta \in \Omega,\\
\\
G_a(x,\zeta) + 8 \ln\left(|x-\zeta|\right),&\quad \zeta\in \pp \Omega.
\end{array}
\right.
\end{equation}

Now, we can state our main results.
\medskip

Our first   result concerns with the existence of solutions   whose concentration points are inside the domain and   approaches different points of the boundary    as $\ep$ goes to zero.

\begin{theo}\label{theo2}
Assume 
\begin{itemize}
\item[(A1)] there exist $m$ different points $\zeta^*_1,\ldots,\zeta^*_m \in \pp \Omega$ such that $\zeta^*_i$ is either a strict local maximum or a strict local minimum   point of $a $ restricted to $\pp \Omega$, satisfying that 
 $$
\pp_{\nu}a(\zeta^*_i ):= \nabla a(\zeta^*_i )\cdot \nu(\zeta^*_i )<0, \quad \forall i=,\ldots,m.
$$\end{itemize}
Then if $\ep>0$ is small enough there exist $m$ points $\zeta^{\ep}_1,\ldots,\zeta^{\ep}_m \in \Omega$, $m$ positive real numbers $d^{\ep}_1,\ldots,d^{\ep}_m$ and a positive solution $u_{\ep}$ of equation \eqref{emdenfowlerhighdim} such that
\begin{equation}\label{profileof thesolutiontheo2}
u_{\ep}(x):=\sum_{i=1}^m \ln\left(\frac{1}{\left(  {d_i^\ep}^2 + |x-\zeta^{\ep}_i|^2\right)^2}\right)+ H_a(x,\zeta^{\ep}_i) + o(1), \quad \forall x \in \overline{\Omega},
\end{equation}
where as $\ep \to 0$ 
$$
\zeta^{\ep}_i \to \zeta^*_i, \quad \rdist(\zeta^{\ep}_i ,\pp \Omega)=\OO\left(|\ln \left({\ep}\right)|^{-1}\right)
$$
and
$$
c\ln \left(\frac{1}{\ep}\right) \leq d^{\ep}_i \leq C\ln \left(\frac{1}{\ep}\right).
$$
for some positive constants $c$ and $C.$\\

In particular, 
 {$$u_\ep(x) \to \sum\limits_{i=1}^m  G_a(x,\zeta^*_i)\ \hbox{in}\ \Omega\setminus\uni\limits_{i=1}^mB(\zeta_i^*,r)\ \hbox{as}\ \ep\to0  $$
for some $r>0$ and
$$\ep^2\int\limits_\Omega e^{u_\ep(x)}dx\to 8\pi m\ \hbox{as}\ \ep\to0.  $$}
\end{theo}

Our second  result concerns with the existence of solutions whose concentration points lie on the boundary    and are far away from each other as $\ep$ goes to zero.

\begin{theo}\label{theo3} 
Assume \begin{itemize}
\item[(A2)] there exist $m$ different points $\zeta^*_1,\ldots,\zeta^*_m\in \pp \Omega$ such that $\zeta^*_i$ is either a strict local maximum or a strict local minimum  point of $a $ restricted to $\pp \Omega.$
\end{itemize}
Then if
 $\ep>0$ is small enough there exist $m$ points $\zeta^{\ep}_1,\ldots,\zeta^{\ep}_m \in \pp \Omega$, $m$ positive real numbers $d^{\ep}_1,\ldots,d^{\ep}_m$ and a positive solution $u_{\ep}$ of equation \eqref{emdenfowlerhighdim} such that
\begin{equation}\label{profileof thesolutiontheo3}
u_{\ep}(x):=\sum_{i=1}^m \ln\left(\frac{1}{\left(  d_i^2 + |x-\zeta^{\ep}_i|^2\right)^2}\right)+ \frac{1}{2}H_a(x,\zeta^{\ep}_i) + o(1), \quad \forall x \in \overline{\Omega},
\end{equation}
where as $\ep \to 0$,
$$
\zeta^{\ep}_i \to \zeta_i^* \hbox{as}\ \ep\to0   \quad \hbox{and} \quad 
c \leq d^{\ep}_i \leq C,
$$
for some positive constants $c$ and $C.$\\

 In particular, 
 {$$u_\ep(x) \to \sum\limits_{i=1}^m  {1\over2} G_a(x,\zeta^*_i)\ \hbox{in}\  \Omega\setminus\uni\limits_{i=1}^mB(\zeta_i^*,r)\ \hbox{as}\ \ep\to0  $$
for some $r>0$ and
$$\ep^2\int\limits_\Omega e^{u_\ep(x)}dx\to 4\pi m\ \hbox{as}\ \ep\to0.  $$}

\end{theo}

Our last existence result concerns with the existence of solutions whose concentration points lie on the boundary and collapse to the same point as $\ep$ goes to zero.

\medskip
\begin{theo}\label{theo4} 
Assume 
\begin{itemize}
\item[(A3)]   $\zeta_0 \in \pp \Omega$ is a strict local maximum point of $a $ restricted to $\pp \Omega$. 
\end{itemize}
 Then for any integer $m\geq 1$ if $\ep>0$ small enough, there exist $m$ points $\zeta^{\ep}_1,\ldots,\zeta^{\ep}_m \in \pp \Omega$, and positive real numbers $d^{\ep}_1,\ldots,d^{\ep}_m$ and a positive solution $u_{\ep}$ of equation \eqref{emdenfowlerhighdim} such that
\begin{equation}\label{profileof thesolutiontheo4}
u_{\ep}(x):=\sum_{i=1}^m \ln\left(\frac{1}{\left(  d_i^2 + |x-\zeta^{\ep}_i|^2\right)^2}\right)+ \frac{1}{2}H_a(x,\zeta^{\ep}_i) + o(1), \quad \forall x \in \overline{\Omega},
\end{equation}
where as $\ep \to 0$,
$$
\zeta^{\ep}_i \to \zeta_0\quad \hbox{and} \quad 
c\ln \left(\frac{1}{\ep}\right) \leq d^{\ep}_i \leq C\ln \left(\frac{1}{\ep}\right),
$$
for some positive constants $c$ and $C.$

 In particular, 
 {$$u_\ep(x) \to {m\over2}  G_a(x,\zeta_0)\ \hbox{in}\ \Omega\setminus B(\zeta_0,r)\ \hbox{as}\ \ep\to0  $$
for some $r>0$ and
$$\ep^2\int\limits_\Omega e^{u_\ep(x)}dx\to 4\pi m\ \hbox{as}\ \ep\to0.  $$}

\end{theo}

All the previous arguments yield immediately the following existence result for the higher-dimensional problem \eqref{p}.

\begin{theo}\label{main}
Assume $D$ is as described in \eqref{D}. 
\begin{itemize}
\item[(i)] If (A1) of Theorem \eqref{theo2} holds,
then for every $\ep>0$ small enough, there exists a $\Gamma-$invariant solution $v_\ep$ to problem \eqref{p} with $m$ {\underline {interior}} layers which concentrate at
$m$ distinct $(N-2)-$dimensional  minimal submanifold  of the boundary of $D $ as $\ep\to0.$
 {Moreover (see \eqref{orbit})
$$\ep^2\int\limits_{\mathcal D} e^{v_\ep(x)}dx\to 8\pi \sum\limits_{i=1}^m |\Xi(\zeta^*_i)|\ \hbox{as}\ \ep\to0.  $$}

\medskip
\item[(ii)] If (A2) of Theorem \eqref{theo3} holds,
then for every $\ep>0$ small enough, there exists a $\Gamma-$invariant solution $v_\ep$ to problem \eqref{p} with $m$ {\underline {boundary}} layers which concentrate at
$m$ distinct $(N-2)-$dimensional  minimal submanifold  of the boundary of $D $ as $\ep\to0.$
 {Moreover (see \eqref{orbit})
$$\ep^2\int\limits_{\mathcal D} e^{v_\ep(x)}dx\to 4\pi \sum\limits_{i=1}^m|\Xi(\zeta^*_i)|\ \hbox{as}\ \ep\to0.  $$}

\medskip
\item[(iii)] If (A3) of Theorem \eqref{theo4},
then for any integer $m\ge1$ and every $\ep>0$ small enough, there exists a $\Gamma-$invariant solution $v_\ep$ to problem \eqref{p} with $m$ {\underline {boundary}} layers which concentrate at
the same $(N-2)-$dimensional  minimal submanifold  of the boundary of $D $ as $\ep\to0.$
 {Moreover (see \eqref{orbit})
$$\ep^2\int\limits_{\mathcal D} e^{v_\ep(x)}dx\to 4\pi m|\Xi(\zeta_0)|\ \hbox{as}\ \ep\to0.  $$}

\end{itemize}
\end{theo}

Let us make some comments.\\

First, we strongly believe that our results hold true even if we drop the symmetry assumption. In particular, we conjecture that (i) and (ii) of Theorem \ref{main} can be rephrased in the more general form

\begin{itemize}
\item[(1)]
  {\it if $D$ is a general bounded domain in $\R^N$ with $N\ge3$ and $\Xi$ is a $(N-2)-$dimensio\-nal minimal submanifold (possibly non-degenerate) of the boundary of $D$ with a suitable sign on the sectional curvatures, then   problem  \eqref{p}
has a solution with an interior layer concentrating along $\Xi$ as $\ep$ goes to zero.}\\
\item[(2)]
  {\it if $D$ is a general bounded domain in $\R^N$ with $N\ge3$ and $\Xi$ is a $(N-2)-$dimensio\-nal minimal submanifold (possibly non-degenerate) of the boundary of $D$, then   problem  \eqref{p}
has a solution with a boundary layer   concentrating along $\Xi$ as $\ep$ goes to zero.}\\
\end{itemize}

Second, the proof of our result relies on a well known Lyapunov-Schmidt procedure. The same strategy has been used by Wei, Ye and Zhou \cite{WEIYEZHOU1,WEIYEZHOU2} to find concentrating solutions for the anisotropic Dirichlet problem
$$-\rdiv(a(x)\nabla u)  = \ep^2 a(x)e^{u}  \quad \hbox{in} \quad \Omega ,\quad \quad u=0 \quad \hbox{on} \quad \pp \Omega,
$$
 where $\Omega\subset \R^2$ is a smooth bounded domain, $a: \overline{\Omega}\to \R$ is a strictly positive smooth function and $\ep>0$ is a small parameter.

\medskip

The structure of the paper is the following. In section \ref{anisotropic robin's function} we perform a careful study of Green's function introduced in \eqref{greensfunctioneqn0}. In section \ref{approximation} we provide the approximation of the solutions predicted by our existence Theorems and compute the error created by this approximation. Section \ref{Reductionscheme} concerns with the finite dimensional reduction scheme which is the first step in the proof of our existence results. In section \ref{energy estimates} we find precise energy estimates for the approximation found in section \ref{approximation}. Finally, in section \ref{proofoftheorems} we provide the detailed proof of our existence Theorems using variational and topological arguments.

\medskip
\section{Anisotropic Green's function}\label{anisotropic robin's function}

In this part we  analyze the asymptotic boundary behavior of the functions $G(\cdot,\zeta)=G_a(\cdot,\zeta)$ and $H(\cdot,\zeta)=H_a(\cdot,\zeta)$ introduced in \eqref{greensfunctioneqn0} and \eqref{regularpartGreen1}, respectively.

\medskip
We begin by recalling some well known facts about Sobolev spaces and refer the reader to \cite{HAIMBREZIS,DINEZZAPALATUCCIVALDINOCI,
GILBARGTRUDINGER} and references therein, for an exhaustive description of these spaces and the related results hereby presented.

\medskip
Let $\Omega\subset \R^2$ be a bounded domain with smooth boundary. The space $L^{p}(\Omega)$ is the space of measurable functions $v:\Omega \to \R$ for which the norm
$$
\|v\|_{L^{p}(\Omega)}:=\left\{
\begin{array}{ccc}
\left(\int_{\Omega}|v(x)|^p\,dx\right)^{\frac{1}{p}},& 1 \leq p < \infty\\
\\
\sup_{x\in \Omega}|v(x)|,& p=\infty
\end{array}
\right.
$$
is finite. The Sobolev space $W^{k,p}(\Omega)$ is the space of functions in $L^p(\Omega)$ having {\it weak derivatives}, up to order $k$, also in $L^p(\Omega)$. The space $W^{k,p}(\Omega)$ is a Banach space endowed with the norm
$$
\|v\|_{W^{k,p}(\Omega)}:=\sum_{i=0}^k
\|D^i v\|_{L^{p}(\Omega)}.
$$ 

Given $k\in \mathbb{N}$, we let $C^{k}(\overline{\Omega})$ denote the space of functions having continuous derivatives of order $k$ up to the boundary. In addition, for any $\A \in (0,1]$, we denote $C^{k,\A}(\overline{\Omega})$ the {\it H\"{o}lder space}, consisting of functions $v \in C^k(\overline{\Omega})$ for which the {\it H\"{o}lder norm}
$$
\|v\|_{C^{k,\A}(\overline{\Omega})}:= \sum_{i=0}^k \|D^i v\|_{L^{\infty}(\Omega)} + \sup \limits_{x,y \in \overline{\Omega}\,\,x\neq y}\frac{|D^i v(x)- D^i v(y)|}{|x-y|^{\A}}
$$ 
is finite.

\medskip
We will make use of the following embeddings for Sobolev functions.  
\begin{equation}\label{SobolevEmdebbings}
W^{2,p}(\Omega) \hookrightarrow \left\{
\begin{array}{ccc}
W^{1,\frac{2p}{2-p}}(\Omega)\cap C^{0,2\left(1-\frac{1}{p}\right)}(\overline{\Omega}),& 1<p<2\\
\\
C^{1,1-\frac{2}{p}}(\overline{\Omega}), & p>2.
\end{array}
\right.
\end{equation}

From the continuity of the {\it trace operator} together with Sobolev embeddings in one dimensional manifolds, we find that for any $1<p<2$,
\begin{equation}\label{traceembedding}
W^{1,p}(\Omega)\hookrightarrow L^q(\pp \Omega), \quad \quad  \forall q \in \left(1,
\frac{p}{2-p}\right). 
\end{equation}

Set
$$
\gamma(x)= \left(\nabla \,\ln a\right)(x), \quad \hbox{for } x\in \overline{\Omega}
$$ 
and notice that $\gamma \in C^{\infty}(\overline{\Omega})$, since $a \in C^{\infty}(\overline{\Omega})$ and $a>0$.

\medskip

Recall also from \eqref{linearpartanisotropiceqn} that
\begin{equation*}
\LL  \,:= \Delta  \,+\, \gamma(x)\cdot \nabla  \,-\, 1
\end{equation*}  
and that $G=G(x,\zeta)$, the {\it Green's function} associated to $\LL$ satisfies   for every $\zeta\in \overline{\Omega}$ the boundary value problem
\begin{eqnarray}
\LL \,G(\cdot,\zeta) + 8\pi\D_   {\zzeta}&=&0\quad \quad  \hbox{in}\quad \Omega, \label{greensfunctioneqn}\\
\nonumber\\
\frac{\pp G(\cdot,\zeta)}{\pp \nu}&=&0\quad \quad  \hbox{on} \quad \pp \Omega.\label{greensfunctionbdrcondition} 
\end{eqnarray}

For $x\in \Omega$, the regular part of $G(x,\zeta)$ is  the function
\begin{equation}\label{regularpartGreen}
H(x,\zeta):= 
\left\{
\begin{array}{ccc}
G(x,\zeta) + 4 \ln \left(|x-\zeta|\right),&\quad 
\zeta \in \Omega,\\
\\
G(x,\zeta) + 8 \ln\left(|x-\zeta|\right),&\quad \zeta\in \pp \Omega
\end{array}
\right.
\end{equation}
Let us introduce the vector function $R=R(z)$, solving 
\begin{equation}\label{ellipticequationR20}
\Delta_z R - R = \frac{z}{|z|^2} \quad \quad \hbox{in } \R^2, \quad \quad R\in L^{\infty}_{loc}(\R^2).
\end{equation}

\medskip
We remark that standard regularity theory implies that $R\in W^{2,p}_{loc}(\R^2)\cap C^{\infty}(\R^2-\{0\})$, for any $p\in (1,2)$. On the other hand, Sobolev embeddings allow us to conclude that for any ball $B_r(0)$ of radius $r>0$ and centered at the origin
\begin{equation}\label{integrabilityR}
R\in W^{1,p}(B_r(0)) \cap C^{0,\frac{1}{p}}(\overline{B_r(0)})
, \quad \quad \forall \,\,p\in(1,\infty).
\end{equation}

Our first result uses the function $R(z)$ to describe the regularity of the family of functions $\zeta \in \overline{\Omega} \mapsto H(\cdot,\zeta)$ and concerns with the local behavior of $H(x,\zeta)$.

\medskip
\begin{theo}\label{RegularPartInner} 
Let $R=R(z)$ be the function described in \eqref{ellipticequationR20}. There exists a function $H_1=H_1(x,\zeta)$, such that 
\begin{itemize}
\item[(i)] for every $\zeta,x\in \overline{\Omega}$,
\begin{equation}\label{regularpartqualitativebehavior}
H(x,\zeta)=H_1(x,\zeta)\,+\,\left\{
\begin{array}{cc}
4\gamma(\zeta)\cdot R(x-\zeta),& \zeta \in \Omega \\
\\
8\gamma(\zeta)\cdot R(x-\zeta),& \zeta \in \pp \Omega
\end{array}
\right.
\end{equation}

and 

\medskip
\item[(ii)] the mapping $
\zeta \in \overline{\Omega} \mapsto H_1(\cdot,\zeta)$
belongs to $C^1\left(\Omega; C^1(\overline{\Omega})\right)\cap C^1(\pp \Omega;C^1(\overline{\Omega}) )$.
\end{itemize}

\medskip

In particular, for any $\A\in (0,1)$, $H\in C^{0,\A}(\overline{\Omega}\times {\Omega})\cap C^{0,\A}(\overline{\Omega}\times \pp {\Omega})$ and the Robin's function $\zeta \in \overline{\Omega} \mapsto H(\zeta,\zeta)$ belongs to $C^{1}(\Omega)\cap C^1(\pp \Omega)$.
\end{theo}

\begin{proof} For $\zeta \in \overline{\Omega}$, let us write 
\begin{equation}\label{constantczeta}
c:=\left\{
\begin{array}{ccc}
4, &\hbox{if } \zeta \in \Omega\\
8, &\hbox{if } \zeta \in \pp \Omega.
\end{array}
\right.
\end{equation}

From \eqref{greensfunctioneqn} and \eqref{regularpartGreen}, we observe that
\begin{equation}
\LL \, H(x,\zeta)= c\gamma(x)\cdot\frac{x-\zeta}{|x-\zeta|^2}\,-\, c\ln(|x-\zeta|), \quad \quad \forall\,\, x\in \Omega\label{EqnregularpartGreen2}
\end{equation}
with the boundary condition
\begin{equation}
\frac{\pp H(x,\zeta)}{\pp \nu_x}= c\nu(x)\cdot\frac{x-\zeta}{|x-\zeta|^2}\,, \quad \quad \forall\,\, x\in \pp \Omega, \quad x\neq \zeta
\label{EqnregularpartGreen2boundary}.
\end{equation}

The right hand side in \eqref{EqnregularpartGreen2} can be written as
\begin{equation}\label{righthand29}
c\gamma(\zeta)\cdot \frac{x-\zeta}{|x-\zeta|^2} + E(x,\zeta),
\end{equation}
where
$$
E(x,\zeta):=c\left(\gamma(x) -\gamma(\zeta)\right)\cdot \frac{x-\zeta}{|x-\zeta|^2} -c\ln(|x-\zeta|).
$$

Using a smooth extension of the function $\gamma$ to a larger compact domain containing $\Omega$, we find a constant $C>0$ depending only on $\gamma$ and $\Omega$ such that
$$
\left|\left(\gamma(x)-\gamma(\zeta)\right)\cdot \frac{(x-\zeta)}{|x-\zeta|^2}\right|\leq C, \quad \hbox{for } x\in \Omega.
$$

On the other hand, given $p\in (0,1)$, there exists a constant $C=C(p,\Omega)>0$, such that for every $\zeta \in \overline{\Omega}$,
$$
\int_{\Omega}\left|\ln(|x-\zeta|)\right|^p \,dx \leq C\int_{0}^{2{\rm diam}(\Omega)}r|\ln(r)|^p\,dr \leq C.
$$
 
We conclude that for any $p\in (1,\infty)$, the mapping
$$
\zeta \in \overline{\Omega} \mapsto E(\cdot,\zeta)\in L^p(\Omega)
$$
is well defined. The Dominated Convergence Theorem yields that $\zeta \mapsto E(\cdot,\zeta)$ belongs to $C({\Omega}; L^p(\Omega))\cap C(\pp {\Omega}; L^p(\Omega))$.

\medskip
Next, let $I_{2\times 2}$ be the $2 \times 2$ identity matrix and 
$$
(x-\zeta)\otimes (x-\zeta):= 
\left[
\begin{array}{cc}
(x_1-\zeta_1)^2& (x_1-\zeta_1) (x_2-\zeta_2)\\
(x_1-\zeta_1)(x_2-\zeta_2)& (x_2-\zeta_2)^2
\end{array}
\right].
$$

We compute for $x,\zeta \in \overline{\Omega}$, $x\neq \zeta$
$$
\nabla_   {\zzeta}E(x,\zeta)=
$$
$$
c\left( D \gamma(\zeta)\cdot \frac{x-\zeta}{|x-\zeta|^2} + \frac{4(\gamma(x)-\gamma(\zeta)}{|x-\zeta|^2}\cdot\left(I_{2\times 2} - 2\,\frac{(x-\zeta)\otimes (x-\zeta)}{|x-\zeta|} \right)\right).
$$

\medskip
Using again Dominated Convergence Theorem again we obtain that for any $p\in (1,2)$, $\zeta\in \overline{\Omega} \mapsto E(\cdot,\zeta)$ belongs to $C^1({\Omega};L^p(\Omega))\cap C^1(\pp {\Omega};L^p(\Omega))$.

\medskip

Define
$$
H_1(x,\zeta):= H(x,\zeta) - c\,\gamma(\zeta)\cdot R(x-\zeta), \quad \quad \forall  x,\zeta\in \Omega,
$$
where $R=R(z)$ is the vector function described in \eqref{ellipticequationR20} and $c$ is described in \eqref{constantczeta}. 

\medskip

From \eqref{ellipticequationR20}, \eqref{EqnregularpartGreen2} and \eqref{righthand29}, we compute the equation for $H_1(\cdot,\zeta)$ to obtain that 
\begin{equation}
\LL \, H_1(x,\zeta)= -c\,\gamma(x)\cdot \left(\gamma(\zeta)\cdot D_{z} R(x-\zeta)\right)\,+\,E(x,\zeta), \quad \quad \hbox{in }\Omega\label{EqnregularpartGreen1.1}
\end{equation}
and the boundary condition reads as 
\begin{equation}
\frac{\pp H_1(x,\zeta)}{\pp \nu_x}\,=\,c\,\nu(x)\cdot\left(\frac{x-\zeta}{|x-\zeta|^2} - \gamma(\zeta)\cdot D_z R(x-\zeta)\right)\quad \quad \hbox{on }\pp \Omega \label{EqnregularpartGreen2.1}.
\end{equation}

The fact that $\zeta \mapsto E(\cdot,\zeta)$ belongs to $C({\Omega};L^p(\Omega))\cap C(\pp {\Omega};L^p(\Omega))$ for any $p \in (1,\infty)$, together with \eqref{integrabilityR}, imply that for any $p\in (1,\infty)$ and any $\zeta \in \overline{\Omega}$, the right hand side in equation \eqref{EqnregularpartGreen1.1} belongs to $L^p(\Omega)$.

\medskip
In the case $\zeta \in \Omega$, the right hand side in \eqref{EqnregularpartGreen2.1} is smooth. In the case $\zeta\in \pp \Omega$, we appeal to Lemma \ref{boundaryterm} in the Appendix and embedding \eqref{traceembedding} to find that for any $\zeta \in \pp {\Omega}$ the right hand side in \eqref{EqnregularpartGreen2.1} belongs to  $L^p(\pp \Omega)$ for any $p>1$.

\medskip

Standard elliptic regularity theory implies that for any $p\in (1,\infty)$, $H_1(\cdot,\zeta)\in W^{2,p}(\Omega)$ and the Sobolev embeddings in \eqref{SobolevEmdebbings} yield that $H_1(\cdot,\zeta)\in C^{1,\A}(\overline{\Omega})$, for any $\A\in (0,1)$.

\medskip

Finally, we check that $\zeta \mapsto H_1(\cdot,\zeta)$ belongs to $C^1(\Omega;C^1(
\overline{\Omega}))\cap C^1(\pp \Omega;C^1(\overline{\Omega}))$. We first deal with the inner regularity. Recall that for any $p\in (1,2)$, $R\in W^{2,p}_{loc}(\R^2)$ and $\zeta \mapsto \nabla_   {\zzeta} E$ belongs $C({\Omega};L^p(\Omega))$. 

\medskip
A direct application of the Dominated Convergence Theorem yields that the mapping
$$
\zeta \in {\Omega} \mapsto \nabla_   {\zzeta}\left[\,-c\,\gamma(x)\cdot \left(\gamma(\zeta)\cdot D_{z} R(x-\zeta)\right)\,+\,E(x,\zeta)\,\right]
$$
belongs to $C({\Omega};L^p(\Omega))$ and consequently, the mapping $\zeta \in \Omega\mapsto\nabla_   {\zzeta} H_1(\cdot,\zeta)\in W^{2,p}(\Omega)$ is well defined and solves 
\begin{equation*}
\LL \, \left(\nabla_   {\zzeta} H_1(x,\zeta)\right)= \nabla_   {\zzeta}\left[\,-c\,\gamma(x)\cdot \left(\gamma(\zeta)\cdot D_{z} R(x-\zeta)\right)\,+\,E(x,\zeta)\,\right]
, \quad \quad \hbox{in }\Omega\label{EqnregularpartGreen1.2}
\end{equation*}
with the boundary condition
\begin{equation*}
\frac{\pp \left(\nabla_   {\zzeta}H_1(x,\zeta)\right)}{\pp \nu_x}\,=\,c\nabla_   {\zzeta}\left(\nu(x)\cdot\left(\frac{x-\zeta}{|x-\zeta|^2} - \gamma(\zeta)\cdot D_z R(x-\zeta)\right) \right)\quad \quad \hbox{on }\pp \Omega \label{EqnregularpartGreen2.2}.
\end{equation*}

Regularity theory and Sobolev embeddings in \eqref{SobolevEmdebbings} and \eqref{traceembedding} imply that the mapping $\zeta \in \Omega \mapsto \nabla_   {\zzeta}H_1(\cdot,\zeta)$ belongs to $C(\Omega;C^{0,\A}(\Omega))$ for any $\A\in (0,1)$. 

\medskip
As for the boundary regularity, we proceed in the same way as we did for the inner regularity, replacing $\nabla_   {\zzeta}$ by its the tangential component respect to the $\pp \Omega$. This concludes the proof of the lemma.
\end{proof}

\medskip
Next, we introduce some notation that will be needed for subsequent developments. Fix $\eta>0$ small such that every $\zeta \in \Omega$ with $\rdist(\zeta,\pp \Omega)< \eta$, has a well defined reflection across $\pp \Omega$ along the normal direction, $\zeta^* \in \Omega^c$. Denote, 
$$
\Omega_{\eta}:= \{\zeta\in \Omega\,:\, \rdist(\zeta,\Omega)<\eta\}
$$
which is also a smooth domain. Observe that for any $\zeta \in \Omega_{\eta}$, $|\zeta- \zeta^*|=2\rdist(\zeta, \pp \Omega)$.

\medskip

Our next result concerns the boundary asymptotic behavior of the Robin's function, which we recall is given by $\zeta \in \overline{\Omega} \mapsto H(\zeta,\zeta)$. 

\begin{prop} \label{asymptoticsof H}
There exists a mapping $z\in C(\Omega_{\eta};C^{0,\A}(\overline{\Omega}))\cap L^{\infty}(\Omega_{\eta};C^{0,\A}(\overline{\Omega}))$ such that
$$
H(x,\zeta):= -4\ln\left(|x-\zeta^*|\right) + z(x,\zeta), \quad \quad \forall\,\, x,\in \overline{\Omega}, \quad \forall\,\, \zeta \in \Omega_{\eta}.
$$

Even more, for every $\zeta \in \Omega_{\eta}$ and $x\in \Omega$
\begin{equation}\label{zeta}
z(x,\zeta)=4\gamma(\zeta)\cdot R(x- \zeta) - 4\gamma(\zeta^*)\cdot R(x-\zeta^*) + \tilde{z}(x,\zeta),
\end{equation}
where the mapping $\zeta \in \Omega_{\eta} \mapsto\tilde{z}(\cdot,\zeta)$ belongs to $C^1\left(\overline{\Omega_{\eta};} C^1(\overline{\Omega})\right).$
\end{prop}

\begin{proof} Consider the function 
$$
z(x,\zeta):= H(x,\zeta) +4\ln\left(|x-\zeta^*|\right), \quad \quad \forall\,\, x,\in \overline{\Omega}, \quad \forall\,\, \zeta \in \Omega_{\eta} ,\quad x\neq \zeta.
$$

We directly compute from \eqref{EqnregularpartGreen2} and \eqref{righthand29}, to find that
\begin{equation}
\LL \, z(\cdot,\zeta)= 4\gamma(x)\cdot\left[\frac{x-\zeta}{|x-\zeta|^2}-\frac{x-\zeta^*}{|x-\zeta^*|^2}\right]\,-\, 4\left[\ln(|x-\zeta|)- \ln(|x-\zeta^*|)\right], \quad \quad \forall\,\, x\in \Omega\label{EqnResidueregularpartGreen}.
\end{equation}
with the boundary condition
\begin{equation}
\frac{\pp z(x,\zeta)}{\pp \nu_x}= 4\nu(x)\cdot\left[\frac{x-\zeta}{|x-\zeta|^2}- \frac{x-\zeta^*}{|x-\zeta^*|^2}\right] \quad \quad \forall\,\, x\in \pp \Omega
\label{EqnResidueregularpartGreen2boundary}.
\end{equation}

The right hand side in  equation \eqref{EqnResidueregularpartGreen} can be written as
$$
4\gamma(\zeta)\cdot\frac{x-\zeta}{|x-\zeta|^2} -4\gamma(\zeta^*)\cdot\frac{x-\zeta^*}{|x-\zeta^*|^2} + \tilde{E}(x,\zeta),
$$
where
\begin{multline}
\tilde{E}(x,\zeta):= 4\left(\gamma(x)-\gamma(\zeta)\right)\cdot \frac{x-\zeta}{|x-\zeta|^2} - 4\left(\gamma(x)-\gamma(\zeta^*)\right) \cdot\frac{x-\zeta^*}{|x-\zeta^*|^2}
\\
 \,-\, 4\left[\ln(|x-\zeta|)- \ln(|x-\zeta^*|)\right], \quad \quad \forall\,\zeta \in \Omega_{\eta}, \quad x\in \Omega. 
\end{multline}

\medskip
Proceeding in the same fashion as in the proof of Theorem \ref{RegularPartInner}, we obtain that the mapping $\zeta \in \overline{\Omega_{\eta}} \mapsto \tilde{E}(\cdot,\zeta)$ belongs to $C^1(\overline{\Omega_{\eta}};L^p(\Omega))$ for any $p\in (1,\infty)$. 

\medskip
To justify \eqref{zeta}, we use the function $R=R(z)$, from \eqref{ellipticequationR20}. We decompose $z(x,\zeta)$ as
$$
z(x,\zeta)=4\gamma(\zeta)\cdot R(x-\zeta) - 4\gamma(\zeta^*)\cdot R(x-\zeta^*) + \tilde{z}(x,\zeta)
$$
to find that if $x\in\Omega$
\begin{equation}
\LL \, \tilde{z}(x,\zeta)=-4\,\gamma(x)\cdot \left(\gamma(\zeta)\cdot D_{z} R(x-\zeta)\,-\,\gamma(\zeta^*)\cdot D_{z} R(x-\zeta^*)\right)+\tilde{E}(x,\zeta)\label{EqnResidueregularpartGreen1}.
\end{equation}
with the boundary condition if $x\in\pp\Omega$
\begin{equation}\frac{\pp \tilde{z}(x,\zeta)}{\pp \nu_x}=
 4\nu(x)\cdot\left[\frac{x-\zeta}{|x-\zeta|^2}- \frac{x-\zeta^*}{|x-\zeta^*|^2}- \left(\gamma(\zeta)\cdot D_{z} R(x-\zeta)\,-\,\gamma(\zeta^*)\cdot D_{z} R(x-\zeta^*)\right)\right].
\label{EqnResidueregularpartGreen2boundary1}
\end{equation}

From \eqref{EqnResidueregularpartGreen1} and \eqref{EqnResidueregularpartGreen2boundary1}, proceeding again as in the proof of Theorem \ref{RegularPartInner}, we obtain that the mapping 
$\zeta\in \overline{\Omega_{\eta}} \mapsto \tilde{z}(\cdot,\zeta)$ belongs to $C^1(\overline{\Omega_{\eta}}; C^1(\overline{\Omega}))$. This concludes the proof  of the proposition.   
\end{proof}

\medskip
{\bf Remark:} For further develoments we notice that from  Proposition \ref{asymptoticsof H} the function $H=H(x,\zeta)$ has continuous partial derivatives in the set $\Omega\times \Omega \setminus \{(x,\zeta)\,:\, x\neq \zeta\}$.

\medskip
Also, directly from Proposition \ref{asymptoticsof H} we obtain the following corollary.

\medskip
\begin{corol}\label{corolario}
Under the assumptions in Proposition 2.1, the Robin's function, 
$$
\zeta \in \Omega \mapsto H(\zeta,\zeta)$$  
satisfies that
\begin{equation}
{H}(\zeta,\zeta):= -4\ln\left(\rdist(\zeta,\pp \Omega)\right) + {\rm z}(\zeta), \quad \quad \forall \,\zeta \in \Omega_{\eta},\label{asymptoticsrobinfunction}
\end{equation}
where ${\rm z}\in C^1(\overline{\Omega_{\eta}})$ and
$$
{\rm z}(\zeta):= 4\gamma(\zeta)\cdot R(0) - \gamma(\zeta^*)\cdot R(\zeta - \zeta^*)+ \tilde{z}(\zeta,\zeta), \quad \quad \forall \,\,\zeta\in \Omega_{\eta}.
$$
\end{corol}

%%%%%%%%%%%%%%%%%%%%%%%%%%%%%%%%%%%%%%%%%%%%%%%%%%%%%%%%%%%%%%%
%%%%%%%%%%%%%%%%%%%%%%%%%%%%%%%%%%%%%%%%%%%%%%%%%%%%%%%%%%%%%%%%

\medskip
\section{The approximation of the solution} \label{approximation}
In this part we find an appropriate aproximation for a solution of equation \eqref{emdenfowlerhighdim} which will allow us to carry out a reduction procedure. We also compute the error created by the choice of our approximation.

\medskip
To set up our approximation, we first identify the formal limit problem associated to equation \eqref{emdenfowlerhighdim}.

\medskip
Dividing by $a(x)$, equation \eqref{emdenfowlerhighdim} becomes 
\begin{equation}
\label{emdenfowlerneumannbdcond2}
\Delta u + \gamma(x)\cdot\nabla u - u + \ep^2 \, e^{u}=0 \quad \hbox{in }\Omega,\qquad
\frac{\pp u}{\pp \nu} =0 \quad \hbox{on }\pp \Omega.
\end{equation}

Take $\ep>0$ 
and set $\Omega_{\ep}:= \ep^{-1}\Omega$. If $u$ is a solution of \eqref{emdenfowlerhighdim}, the function
\begin{equation}\label{rescaling}
v(y)=4\,\ln(\ep) + u(\ep y), \quad \quad y\in \Omega_{\ep} 
\end{equation}
solves the equation
\begin{equation}
\label{emdenfowlerneumannbdcond22}
\Delta v + \ep \gamma(\ep y)\cdot\nabla v - \ep^2(v-4\ln(\ep)) + \, e^{v}=0 \quad \hbox{in }\Omega_{\ep},\quad \quad \frac{\pp v}{\pp \nu_{\ep}} =0 \quad \hbox{on }\pp \Omega_{\ep}
\end{equation}
where $\nu_{\ep}$ is the inner unit normal vector to $\pp \Omega_{\ep}$.

\medskip
From \eqref{emdenfowlerneumannbdcond2},  formally as $\ep\to 0^+$, we obtain the limit equation    
\begin{equation}\label{LiouvilleEqnR2}
\Delta V + e^V=0, \quad \hbox{in }\R^2, \quad \quad \nabla V \in L^2(\R^2).
\end{equation}

\medskip
Solutions to  \eqref{LiouvilleEqnR2} are given by
\begin{equation}\label{LiouvilleEqnnoscaling}
V(y):= \ln
\left(\frac{8\,d^2}{\left(d^2 + |y-\zeta'|^2\right)^2}\right), \quad \hbox{for } y \in \R^2
\end{equation}
where $d \in \R$ and $\zeta'\in \R^2$ are arbitrary parameters.

\medskip

Pulling back the rescaling in \eqref{rescaling}, for any $d>0$ and any $\zeta\in \R^2$, the function
\begin{equation}\label{LiouvilleEqnrescaling}
U_{d,\zeta}(x):= \ln
\left(\frac{8\,d^2}{\left(\ep^2 d^2 + |x-\zeta|^2\right)^2}\right), \quad \hbox{for } x \in \R^2
\end{equation}
solves the equation 
\begin{equation}\label{LiouvilleEqnR2rescaling}
\Delta U_{d,\zeta} + \ep^2e^{U_{d,\zeta}}=0, \quad \hbox{in }\R^2, \quad \quad \nabla_x U_{d,\zeta} \in L^2(\R^2).
\end{equation}

Let $m\in \mathbb{N}$ be fixed. Consider $m$ real numbers $d_i>0$ and $m$ arbitrary different points $\zeta_i \in \overline{\Omega}$. For every $i=1,\ldots,m$, define
\begin{equation}\label{LiouvilleEqnrescalingwithparameters}
U_i(x):= U_{d_i,\zeta_i}(x)=\ln\left(\frac{8d^2_i}{\left(\ep^2 \,d_i^2 + |x-\zeta_i|^2\right)^2}\right),\quad \quad  x\in \R^2. 
\end{equation}

Let $PU_{i}\in H^1(\Omega)$ be the solution of
\begin{equation}\label{EquationPU}
\Delta PU_{i}
+ \gamma(x) \cdot \nabla PU_{i} - PU_{i}+\ep^2e^{U_{i}}=0 \quad \hbox{in }\Omega, 
\qquad
\frac{\pp PU_{i}}{\pp \nu}=0 \quad \hbox{on }\pp \Omega.  
\end{equation}

By standard regularity theory, $PU_{i}\in C^{\infty}(\overline{\Omega})$, so that $PU_{i}$ is indeed a classical solution of \eqref{EquationPU}.

\medskip
Observe that each function $PU_{i}$ depends on $\ep>0$, $d_i$ and $\zeta_i$, but for notational simplicity, we unify this depedence using the subindex $i$.

\medskip
For every $i=1,\ldots,m$,  consider the function $H_{i}\in C^{\infty}(\overline{\Omega})$ given by
$$
H_{i}(x):=PU_{i}(x)- U_{i}(x), \quad \quad \hbox{for}\quad x\in \overline{\Omega}.
$$ 

From \eqref{EquationPU}, $H_{i}$ solves the equation 
\begin{equation}\label{EquationHdzeta}
\Delta H_{i} + \gamma(x)\cdot \nabla H_{i} - H_{i} = U_{i} - \gamma(x)\cdot \nabla U_{i} \quad \hbox{in }\Omega
\end{equation}
with the boundary condition
\begin{equation}
\label{BdConditionHdzeta}
\frac{\pp H_i}{\pp \nu} = -\frac{\pp U_{i}}{\pp \nu} \quad \hbox{on }\pp \Omega.
\end{equation}

\medskip
The following assumptions on the parameters $d_i$ and $\zeta_i$ will play a crucial role in what follows.

\medskip
We assume that for every $i=1,\ldots,m$, the parameters $d_i>0$ and $\zeta_i\in \overline{\Omega}$ depend on $\ep>0$. This dependence is expressed by the conditions:
\begin{equation}\label{assumptionsdi}
\lim_{\ep\to 0^+}\ep\, d^{\A}_i =0, \quad \quad \forall\, \A>0
\end{equation}
and in the case that $\zeta_i \in \Omega$, for some $c_0>0$ and for some $\kappa\geq 1$
\begin{equation}\label{assumptionszetai}
\rdist(\zeta_i,\pp \Omega)\geq c_0|\ln(\ep)|^{-\kappa}. 
\end{equation}

\medskip
Next, lemma concerns the asymptotic behavior of the functions $H_{i}$ in terms of $d_i$, $\zeta_i$ and $\ep>0$ small enough. For every fixed $i=1,\ldots,m$ and $\zeta_i\in \overline{\Omega}$, we will use the convention that 
$$
c_i:=
\left\{
\begin{array}{ccc}
1, & \zeta_i \in \Omega\\
\frac{1}{2}, &\zeta_i \in \pp \Omega.
\end{array}
\right.
$$

\medskip

\begin{lemma}\label{expansionHdzeta}
Assume conditions \eqref{assumptionsdi} and \eqref{assumptionszetai}. Then, for every $i=1,\ldots,m$ and every $\ep>0$ small enough, there exists a function $z_i$ such that
\begin{itemize}
\item[(i)] for every $x\in \Omega$ 
\begin{equation}\label{AsymptResiduePU}
H_{i}(x)= -\ln(8d_i^2) + c_i\,H(x,\zeta_i) + z_{i}(x),
\end{equation}
where $H=H_a$ is defined in \eqref{regularpartGreen1}
and
\medskip
\item[(ii)]  $\forall \,p\in(1,2)$, $z_i \in W^{2,p}(\Omega) \cap C(\overline{\Omega})$ and  
$$
\|z_{i}\|_{W^{2,p}(\Omega)}\,+\,\|z_{i}\|_{L^{\infty}(\Omega)}\leq C\ep^{\frac{1}{p}}d_i^{\frac{1}{p}},
$$
where the constant $C>0$ depedens only on $p$.
\end{itemize}
\end{lemma}

\begin{proof} For notational simplicity, throughout this proof, we omit the subindex $i$. Let $\ep>0$ be small, $d>0$ and $\zeta \in \overline{\Omega}$. 

\medskip

For $x\in \Omega$, we have that $H(x):= PU(x) - U(x)$. We set for $x\in \overline{\Omega}$
$$
z(x) := H(x) + \ln(8d^2) - 
\left\{
\begin{array}{ccc}
H(x,\zeta), & \hbox{if } \quad \zeta \in \Omega,\\
\\
\frac{1}{2}H(x,\zeta), & \hbox{if } \quad \zeta \in \pp \Omega.
\end{array}
\right.
$$ 

\medskip
Recall from \eqref{linearpartanisotropiceqn} that $\LL= \Delta + \gamma(x)\cdot \nabla -1$. From \eqref{EqnregularpartGreen2} and \eqref{EquationHdzeta}, the equation for $z$ reads as 
$$
\LL\, z(x) = 2\ln\left(\frac{|x-\zeta|^2}{\ep^2 d^2 + |x-\zeta|^2}\right) + 4\gamma(x)\cdot \frac{(x-\zeta)}{|x-\zeta|^2}\cdot \frac{\ep^2 d^2}{\ep^2d^2 + |x-\zeta|^2} \quad \hbox{in }\Omega,
$$
$$
\frac{\pp z}{\pp \nu} = -4\nu(x)\cdot \frac{(x-\zeta)}{|x-\zeta|^2}\cdot \frac{1}{\ep^2d^2 + |x-\zeta|^2} \quad \hbox{on }\pp \Omega.
$$

\medskip

For any $p>1$, we estimate
\begin{eqnarray*}
\int_{\Omega}\left|\ln\left(\frac{|x-\zeta|^2}{\ep^2 d^2 + |x-\zeta|^2}\right)\right|^pdx &\leq & C\int_{0}^{2\, {\rm diam}(\Omega)}r\left|\ln\left(\frac{r^2}{\ep^2 d^2 +r^2}\right)\right|^p dr\\
&=& C\ep^2 d^2 \int_{0}^{\frac{2{\rm diam}(\Omega)}{\ep\,d}}r\left|\ln\left(1+\frac{1}{r^2}\right)
\right|^pdr\\
&\leq &C \ep^2 d^2\, \left(\int_0^{1}r\left|\ln\left(1+\frac{1}{r^2}\right)
\right|^pdr+\int_{1}^{\infty}r^{1-2p}dr\right)\\
&\leq& C \ep^2 d^2.
\end{eqnarray*}

\medskip

Hence, we obtain that
\begin{equation}\label{EQNrighthandsideI}
\left\|\ln\left(\frac{|x-\zeta|^2}{\ep^2 d^2 + |x-\zeta|^2}\right)\right\|_{L^p(\Omega)} \leq  C \ep^{\frac{2}{p}} d^{\frac{2}{p}}.
\end{equation}

\medskip
As for the second term, let us take $p\in(1,2)$, so that 
\begin{eqnarray*}
\int_{\Omega}\left|\frac{\gamma(x)\cdot (x-\zeta)}{|x-\zeta|^2}\cdot \frac{\ep^2 d^2}{\ep^2d^2 + |x-\zeta|^2}\right|^p dx&\leq & C\ep^{2-p} d^{2-p} \int_{0}^{\frac{2{\rm diam}(\Omega)}{\ep d}} \frac{r^{1-p}}{(1+r^2)^p}dr\\
&\leq &C\ep^{2-p} d^{2-p} \left(\int_{0}^{1} \frac{r^{1-p}}{(1+r^2)^p}dr+\int_1^{\infty}r^{1-p}dr
\right)
\\
&\leq &  C\ep^{2-p} d^{2-p}
\end{eqnarray*}
and therefore
\begin{equation}\label{EQNrighthandsideII}
\left\|\frac{\gamma(x)\cdot (x-\zeta)}{|x-\zeta|^2}\cdot \frac{\ep^2 d^2}{\ep^2d^2 + |x-\zeta|^2}\right\|_{L^p(\Omega)}\leq C \ep^{\frac{2}{p}-1}d^{\frac{2}{p}-1}.
\end{equation}

As for the boundary term, if $\zeta\in \Omega$ we use conditions \eqref{assumptionsdi}  and \eqref{assumptionszetai}, to find that, provided $\ep>0$ and $\ep d$ are small enough 
$$
\left|\frac{\nu(x)\cdot (x-\zeta)}{|x-\zeta|^2}\cdot\frac{\ep^2 d^2}{\ep^2 d^2+ |x-\zeta|^2}\right|\leq \frac{C\ep^2 d^2}{|x-\zeta|^3}\leq C\ep d, \quad \quad \forall \,x\in \pp \Omega,
$$

On the other hand, if $\zeta \in \pp \Omega$, from Lemma \eqref{boundaryterm}, we estimate
$$
\left|\frac{\nu(x)\cdot (x-\zeta)}{|x-\zeta|^2}\cdot\frac{\ep^2 d^2}{\ep^2 d^2+ |x-\zeta|^2}\right|\leq \frac{C\ep^2 d^2}{\ep^2 d^2+ |x-\zeta|^2}, \quad \quad \forall \,x\in \pp \Omega.
$$

Let us take $\D>0$ small but fixed,  so that
$$
\left|\frac{\nu(x)\cdot (x-\zeta)}{|x-\zeta|^2}\cdot\frac{\ep^2 d^2}{\ep^2 d^2+ |x-\zeta|^2}\right|\leq \frac{C\ep^2 d^2}{\D^2}, \quad \quad \forall \,x\in \pp \Omega\cap B^c_{\D}(\zeta)
$$
while for any $p>1$, we estimate
\begin{eqnarray}
\int_{\pp \Omega\cap B_{\D}(\zeta)} \left|\frac{\nu(x)\cdot (x-\zeta)}{|x-\zeta|^2}\cdot\frac{\ep^2 d^2}{\ep^2 d^2+ |x-\zeta|^2}\right|^p dx &\leq & C \ep d \int_{0}^{\frac{\D}{\ep d}}\frac{1}{(1 + s^2)^p}ds\\
&\leq & C \ep d.
\end{eqnarray}

We conclude that for any $\zeta \in \overline{\Omega}$
\begin{equation}
\label{EQNrighthandsideBdCondition}
\left\|\frac{\nu(x)\cdot (x-\zeta)}{|x-\zeta|^2}\cdot\frac{\ep^2 d^2}{\ep^2 d^2+ |x-\zeta|^2}\right\|_{L^p(\pp \Omega)}\leq C\ep^{\frac{1}{p}}d^{\frac{1}{p}}.
\end{equation}

\medskip
Standard elliptic regularity implies that for any $p\in (1,2)$, $z\in W^{2,p}(\Omega)$. The Sobolev embeddings in \eqref{SobolevEmdebbings} together with estimates \eqref{EQNrighthandsideI}, \eqref{EQNrighthandsideII} and \eqref{EQNrighthandsideBdCondition},  imply that
$$
\|z\|_{L^{\infty}(\Omega)}+\|z\|_{W^{2,p}(\Omega)}\leq C\ep^{\frac{1}{p}} d^{\frac{1}{p}}
$$
and this concludes the proof of the lemma.
\end{proof}

Now we are in a position to set our approximation. Consider the parameters $d_1, \ldots,d_m \in \R_+$ and $\zeta_1, \ldots,\zeta_m \in \overline{\Omega}$ satisfying \eqref{assumptionsdi} and \eqref{assumptionszetai} respectively. In addition, assume that 
\begin{equation}\label{logconditiondi}
\ln\left(8 d_i^2\right)= c_iH(\zeta_i,\zeta_i) \,+ \, \sum_{j\neq i} c_jG(\zeta_i,\zeta_j), \quad \quad \forall\,i=1,\ldots,m.
\end{equation}

We set as approximation the function 
\begin{equation}\label{approx}
u_{\ep}(x):=\sum_{i=1}^m PU_i(x)=\sum_{i=1}^m U_i(x)+ H_i(x), \quad \quad \hbox{for }x\in \overline{\Omega}.
\end{equation}

Using the rescaling in \eqref{rescaling}, we also set for every $y\in \Omega_{\ep}$
\begin{eqnarray}
v_{\ep}(y)&:=& 4\,\ln(\ep) + u_{\ep}(\ep y)\label{rescalingapprox}\\
&=& 4(1 - m)\ln(\ep) + \sum_{i=1}^m \ln\left(\frac{8d_i^2}{\left(d_i^2 \,+\, |y -\zeta_i'|^2\right)^2}\right) + H_i(\ep y),\nonumber
\end{eqnarray}
where we have denoted $\zeta_i':= \frac{\zeta_i}{\ep}$. 

For our subsequent developments, we introduce another condition on the numbers $d_1,\ldots,d_m$ and the points $\zeta_1,\ldots,\zeta_m\in \overline{\Omega}$. Consider the real number 
\begin{equation}\label{distancezetaizetaj}
c_{\ep}:=\min\{|\zeta_i -\zeta_j|\,:\, i,j=1,\ldots,m, \quad i\neq j\}
\end{equation}
which is well defined, positive and uniformly bounded above, since $\Omega$ is bounded. We assume in addition that  
\begin{equation}\label{nottoclosezetaij}
\lim \limits_{\ep\to 0^+} \frac{c_{\ep}}{\ep\,d_i}=\infty, \quad \quad \forall \,i=1,\ldots,m.
\end{equation}

Condition \eqref{nottoclosezetaij} means that as $\ep\to 0^+$, the numbers $d_i$ might go to infinity, but at a rate that is controlled by the distance between the points $\zeta_i$.

\medskip

Let us denote
$$
W:= e^{v_{\ep}}, \quad \quad S(v_{\ep}):= \Delta v_{\ep} + \ep \gamma(\ep y)\cdot\nabla v_{\ep} - \ep^2(v_{\ep}-4\ln(\ep)) + \, e^{v_{\ep}}.
$$

Next lemma provides the size of the error term $S(v_{\ep})$. 
\medskip
\begin{lemma}\label{sizeof the errorSvep}
Assume hypotheses in Lemma \ref{expansionHdzeta} and conditions \eqref{logconditiondi} and \eqref{nottoclosezetaij}. Then for any $\A, \beta \in (0,1)$, there exists $\ep_0>0$ small such that for any $\ep\in (0,\ep_0)$, there exists a function $\theta_{\ep}(y)$ such that 
$$
|\theta_{\ep}(y)|\leq C \ep^{3+\A} + C \ep^{\beta} \sum_{i=1}^m |y-\zeta_i'|^{\beta}, \quad \quad \forall\, y\in \Omega_{\ep}
$$
and 
\begin{equation}\label{sizeW}
W(y)= \sum_{i=1}^m \frac{8\,d_i^2}{\left(\ep^2d_i^2 \,+\,|y-\zeta_i'|\right)^2}\left(1 + \theta_{\ep}(y)\right), \quad \quad \forall\, y\in \Omega_{\ep}.
\end{equation}

Even more, 
\begin{equation}\label{sizeSvep}
|S(v_{\ep})(y)|\leq C \ep^{\A}\sum_{i=1}^m \frac{1}{1+ |y-\zeta_i'|^3}, \quad \quad \forall\, y\in \Omega.
\end{equation}
\end{lemma}

\begin{proof}
We proceed as in the proof of expressions (20) and (21) in \cite{DELPINOWEI}. For the sake of completeness we include the detailed computations. 

\medskip
We first prove \eqref{sizeW}. Let $\ep>0$ be small. Observe that for any $i=1,\ldots,m$ and any $y\in \Omega_{\ep}$
\begin{eqnarray*}
v_{\ep}(y)&=&4\ln(\ep) + U_i(\ep y) + H_i(\ep y)+ \sum_{j\neq i}\left(U_j(\ep y) + H_j(\ep y)\right)\\
&=&\ln\left(\frac{8d^2_i}{(d^2_i + |y-\zeta'_i|^2)^2}\right)+ H_i(\ep y)+ \sum_{j\neq i}\left(U_j(\ep y) + H_j(\ep y)\right).
\end{eqnarray*}

Using \eqref{AsymptResiduePU}, we find that
\begin{eqnarray*}
v_{\ep}(y)&=& \ln\left(\frac{8d^2_i}{(d^2_i + |y-\zeta'_i|^2)^2}\right)- \ln(8 d_i^2) + c_i H(\ep y, \zeta_i) + z_i(\ep y) +\\
&&+ \sum_{j\neq i}\left(\ln\left(\frac{8d^2_j}{(\ep^2 d^2_j + |\ep y-\zeta_j|^2)^2}\right)-\ln(8 d_j^2) + c_jH(\ep y, \zeta_j) + z_j(\ep y)\right).
\end{eqnarray*}

Therefore, for any $\A\in (0,1)$,
\begin{multline}\label{vepfirstasymptotics}
v_{\ep}(y)=\ln\left(\frac{1}{(d^2_i + |y-\zeta'_i|^2)^2}\right)+ c_i H(\ep y, \zeta_i)\\
+ \sum_{j\neq i}\left(\ln\left(\frac{1}{(\ep^2 d^2_j + |\ep y-\zeta_j|^2)^2}\right) + c_j H(\ep y, \zeta_j)\right)
+\sum_{j=1}^m\OO_{L^{\infty}(\Omega_{\ep})}(\ep^{\A}d_j^{\A}).
\end{multline}

To estimate more accurately \eqref{vepfirstasymptotics}, we first notice from Theorem \ref{RegularPartInner} that for any fixed $j=1,\ldots,m$ and any $\beta\in (0,1)$,
\begin{equation}\label{HolderregularpartGreensFnct}
H(\ep y,\zeta_j)=H(\zeta_i,\zeta_j)+ \OO_{L^{\infty}(\Omega_{\ep})}(|\ep y - \zeta_i|^{\beta}), \quad \quad \forall\, y\in \Omega_{\ep}.
\end{equation}

Next, we consider two regimes: the first one close to the point $\zeta_i'$ and the second one when we are far from $\zeta_i'$. To be more precise, take any $\tilde{\A}\in (0,1)$ and $\D_{\ep}\in (0,(1-\tilde{\A}) c_{\ep}]$, where $c_{\ep}$ is defined in \eqref{distancezetaizetaj}. Notice that for any $j\neq i$ and any $y\in B_{\frac{\D_{\ep}}{\ep}}(\zeta'_i)$
\begin{eqnarray*}
|y-\zeta'_j|
\geq  \frac{|\zeta_i-\zeta_j|}{\ep} - |y-\zeta'_i| \geq \frac{\tilde{\A}\,|\zeta_i-\zeta_j|}{\ep},
\end{eqnarray*}
so that $|\ep y -\zeta_j|\geq \tilde{\A} c_{\ep}$.

\medskip
Using conditions \eqref{assumptionsdi} and \eqref{nottoclosezetaij}, for any $y\in B_{\frac{\D_{\ep}}{\ep}}(\zeta_i')$ and any $j\neq i$, we compute 
\begin{small}
\begin{eqnarray}\label{eqnasymptitoclog}
\ln\left(\frac{1}{(\ep^2 d^2_j + |\ep y-\zeta_j|^2)^2}\right)&=&-4\ln\left(|\ep y-\zeta_j|\right)+ \OO_{L^{\infty}\left(B_{\frac{\D_{\ep}}{\ep}}(\zeta_i')\right)}(\ep^2d_j^2 |y-\zeta'_j|^{-2})\nonumber\\
&=&-4\ln(|\ep y -\zeta_j|) + \OO_{L^{\infty}\left(B_{\frac{\D_{\ep}}{\ep}}(\zeta_i')\right)}(\ep^4 d_j^4)\nonumber\\
&=&-4\ln(|\zeta_i -\zeta_j|) + \OO_{L^{\infty}\left(B_{\frac{\D_{\ep}}{\ep}}(\zeta_i')\right)}\left(\ep^{4-\tilde{\A}} + \ep\,|y-\zeta_i'|\right).
\end{eqnarray}
\end{small}

Therefore, using expressions \eqref{HolderregularpartGreensFnct} and \eqref{eqnasymptitoclog}, we can choose $\tilde{\A}>1-\A$ such that for any $y\in B_{\frac{\D_{\ep}}{\ep}}(\zeta'_i)$ expression \eqref{vepfirstasymptotics} reads as
\begin{eqnarray}
v_{\ep}(y)&=&\ln\left(\frac{1}{(d^2_i + |y-\zeta'_i|^2)^2}\right)+ c_iH(\zeta_i, \zeta_i) + \sum_{j\neq i} \left(-4\ln(|\zeta_i -\zeta_j|)  + c_j\,H(\zeta_i, \zeta_j)\right)
\nonumber\\
&&+ \OO_{L^{\infty}\left(B_{\frac{\D_{\ep}}{\ep}}(\zeta_i')\right)}\left(\ep^{4-\tilde{\A}} + \ep \,d_i|y-\zeta_i'|\right)+ \OO_{L^{\infty}\left(\Omega_{\ep}\right)}(|\ep y - \zeta_i|^{\beta})\nonumber\\
&=&\ln\left(\frac{1}{(d^2_i + |y-\zeta'_i|^2)^2}\right)+ c_iH(\zeta_i, \zeta_i) 
+ \sum_{j\neq i} c_jG(\zeta_i,\zeta_j)\nonumber\\
&&+ \OO_{L^{\infty}\left(B_{\frac{\D_{\ep}}{\ep}}(\zeta_i')\right)}\left(\ep^{3+\A} + \ep^{\beta}|y-\zeta_i'|^{\beta}\right)\nonumber\\
&=& \ln\left(\frac{8d_i^2}{(d^2_i + |y-\zeta'_i|^2)^2}\right)+ \OO_{L^{\infty}\left(B_{\frac{\D_{\ep}}{\ep}}(\zeta_i')\right)}\left(\ep^{3+\A} + \ep^{\beta}|y-\zeta_i'|^{\beta}\right).\label{vepsecondasymptotics}
\end{eqnarray}

Directly from the identities in \eqref{vepsecondasymptotics}, we obtain that in $B_{\frac{\D_{\ep}}{\ep}}(\zeta_i')$
\begin{equation}\label{Wfirstestimate}
e^{v_{\epsilon}(y)}=\frac{8d_i^2}{(d_i^2 + |y-\zeta'_i|^2)^2}\,e^{\OO\left(\ep^{3+\A} + \ep^{\beta}|y-\zeta_i'|^{\beta}\right)}.
\end{equation}

On the other hand, from \eqref{assumptionsdi} and \eqref{nottoclosezetaij}, we can choose $\D_{\ep}$ such that for every $i=1,\ldots,m$, {$\frac{\D_{\ep}}{\ep d_i}$ is bounded below by $c \ep^{1-\A}d_i^{1-\A}$}. Hence, if $|y-\zeta_i'|\geq \frac{\D_{\ep}}{\ep}$, $\frac{|\ep y - \zeta_i|}{\ep d_i}\geq \frac{\D_{\ep}}{\ep d_i}$ and from \eqref{rescalingapprox} we find that that \begin{eqnarray*}
v_{\ep}(y)&=&4\ln(\ep) -\sum_{i=1}^m -4\ln(|\ep y - \zeta_i|) + \OO_{L^{\infty}}(\ep^2 d_i^2 |\ep y - \zeta_i|^{-2})\\
&=& 4\ln(\ep) - \sum_{i=1}^m 4\ln(|\ep y - \zeta_i|) + \OO_{L^{\infty}(\Omega)}(\ep^{\A}d_i^{\A})\\
&=& 4\ln(\ep) + \OO_{L^{\infty}(\Omega)}\left(\ln(\D_{\ep})\right).
\end{eqnarray*}

Hence, for $y\in \Omega_{\ep}-\cup_{i=1}^m B_{\frac{\D_{\ep}}{\ep}}(\zeta_i')$, we have that
$$
e^{v_{\ep}(y)}=\OO_{L^{\infty}(\Omega_{\ep}-\cup_{i=1}^m B_{\frac{\D_{\ep}}{\ep}}(\zeta_i'))}(\ep^{3+\A}).
$$

Denoting by $\chi_{B}$ the characteristic function of a set $B\subset \R^2$, we write for $y\in \Omega_{\ep}$ 
\begin{eqnarray*}
e^{v_{\ep}(y)}&:=&\sum_{i=1}^m \left(e^{v_{\ep}(y)}\chi_{B_{\frac{\D_{\ep}}{\ep}}(\zeta_i')}+e^{v_{\ep}(y)}\chi_{\Omega_{\ep}-\cup_{i=1}^m B_{\frac{\D_{\ep}}{\ep}}(\zeta_i')}\right)\\
&=&\sum_{i=1}^m\frac{8d_i^2}{(d_i^2 + |y-\zeta'_i|^2)^2}\,e^{\OO(\ep^{3+\A}+ |\ep y - \zeta_i|^{\beta})} +\OO(\ep^{3+\A})
\end{eqnarray*}
and asymptotics in \eqref{sizeW} follow.

\medskip
To find estimate \eqref{sizeSvep}, we use the fact that 
\begin{eqnarray*}
S(v_{\ep}(y))&=&e^{v_{\ep}(y)} -\sum_{i=1}^m e^{4\ln(\ep)+ U_i(\ep y)}\\
&=&\sum_{i=1}^m \left(e^{v_{\ep}(y)}-e^{4\ln(\ep)+U_i(\ep y)}\right)\chi_{B_{\frac{\D_{\ep}}{\ep}}(\zeta_i')}\\
&&+ \left(e^{v_{\ep}(y)}-\ep^4\sum_{i=1}^m e^{U_i(\ep y)}\right)\chi_{\Omega_{\ep}-\cup_{i=1}^m B_{\frac{\D_{\ep}}{\ep}}(\zeta_i')}\\
&=&\sum_{i=1}^m \frac{8d_i^2}{(d_i^2 + |y-\zeta_i'|^2)^2}\theta_{\ep}(y) + \OO(\ep^{3 + \A})
\end{eqnarray*} 
from where \eqref{sizeSvep} follows.
\end{proof}

\section{The reduction scheme}\label{Reductionscheme}

In this part we use the approximation \eqref{approx} described in section \ref{approximation} to solve equation \eqref{emdenfowlerhighdim} using a finite dimensional reduction reduction procedure. 

\medskip
We use the convention that $m=k+l$ for some $k,l \in \{0,\ldots,m\}$ and that 
$$
\zeta_1,\ldots,\zeta_k \in \Omega \quad \hbox{and}\quad \zeta_{k+1},\ldots,\zeta_{k+l} \in \pp \Omega.
$$

\medskip

It will be more convenient to work with the rescaling \eqref{rescaling}. Using the function described in \eqref{rescalingapprox}, we look for a solution $V_{\ep}$ of equation \eqref{emdenfowlerneumannbdcond22} having the form
$$
V_{\ep}(y):=v_{\ep}(y) + \phi(y), \quad \quad y\in \Omega_{\ep}
$$
so that $\phi$ must solve the nonlinear boundary value problem \begin{equation}\label{nonlineareqnphi1}
\Delta \phi -\ep^2 \phi = -S(v_{\ep})-e^{v_{\ep}}\phi -\ep \gamma(\ep y) \cdot \nabla\phi- N(\phi) \quad \hbox{in} \quad \Omega_{\ep}
\end{equation}
with the boundary condition
\begin{equation}
\label{boundaryconditionphi1}
\frac{\pp \phi}{\pp \nu}=0 \quad \hbox{on} \quad  \pp \Omega_{\ep},
\end{equation}
where we have denoted 
$$
N(\phi):= e^{v_{\ep}}\left[e^{\phi} -1 - \phi\right].
$$

\medskip
To solve \eqref{nonlineareqnphi1}-\eqref{boundaryconditionphi1}, we follow the developments from section 3 in \cite{DELPINOWEI}. 

\medskip 
For $i=1,\ldots,m$ fixed, we write 
$$
J_i:=\left\{
\begin{array}{ccc}
2,& \hbox{for } 1 \leq i \leq k\\
\\
1,& \hbox{for } k+1 \leq i \leq k+l. 
\end{array}
\right.
$$

Consider the following linear problem: given a function $h \in L^{\infty}(\Omega_{\ep})$, find $\phi \in C^1(\overline{\Omega_{\ep}})$ and constants $c_{ij}$ for $i=1,\ldots,m$ and for $j=1,J_i$ such that
\begin{equation}\label{lineareqanphi1}
\Delta \phi -\ep^2 \phi = -e^{v_{\ep}}\phi + h +
\sum_{i=1}^m \sum_{j=1,J_i}c_{ij}\chi_{ij} Z_{ij} \quad \hbox{in }\quad \Omega_{\ep}
\end{equation}
with the boundary and orthogonality conditions
\begin{equation}
\label{boundaryorthogonalityconditionphi1}
\frac{\pp \phi}{\pp \nu}=0 \quad \hbox{on} \quad  \pp \Omega_{\ep}, \quad \quad \int_{\Omega_{\ep}}\chi_{ij}Z_{ij}\phi =0, \quad \forall\, i=1,\ldots,m, \quad j=1,J_i,
\end{equation}
where the functions $Z_{ij},\chi_{\ij}$ are next defined.

\medskip
For $i=1,\ldots,m$, we set
$$
z_{i0}(y):= \frac{1}{d_i} - \frac{2 d_i}{d_i^2 + |y|^2}, \quad \quad z_{ij}= \frac{y_j}{d_i^2 + |y|^2}, \quad \forall j=1,J_i
$$

For any $i=1,\ldots,m$ fixed, Lemma 2.1 in \cite{ESPOSITOGROSSIPISTOIA} guarantees that the only solutions to 
$$
\Delta \phi + \frac{8d_i^2}{d_i^2 + |y|^2}\phi =0 \quad \hbox{in }\R^2, \quad \quad |\phi(y)|\leq C(1 +|y|^{\sigma}), \quad \hbox{for some }\sigma>0
$$
are the linear  combinations of $z_{ij}(y)$ for $j=0,1,2$. Observe that 
$$
\frac{8d_i^2}{d_i^2 + |y|^2}=e^{V(y)}, \quad y\in \R^2,
$$
where $V$ is the function in \eqref{LiouvilleEqnnoscaling}
with $d=d_i$ and $\zeta'=0$.

\medskip

Next, let $r_0>0$ be a large number and $\chi:\R \to [0,1]$  be a non-negative  smooth cut-off function so that 
$$
\chi(r)=\left\{
\begin{array}{ccc}
1,& \hbox{if} \quad r\leq r_0\\
\\
0,& \hbox{if } \quad r\geq r_0+1.
\end{array}
\right.
$$

For $i=1,\ldots,k$, we have $\zeta_i\in \Omega$ and we define
$$
\chi_i(y):=\chi(|y-\zeta_i'|), \quad \quad Z_{ij}(y):=\chi_i(y)\,z_{ij}(y-\zeta_i') , \quad j=1,2.
$$

As for $i=k+1, \ldots,k+l$, we have that $\zeta_i \in \pp \Omega$. For notational simplicity, assume for the moment that $\zeta_i=0$ and that the inner unit normal vector to $\pp \Omega$ at $\zeta_i$ is the vector ${\rm e}_2=(0,1)$. Hence, there exists $r_1>0$, $\D>0$ small and a function ${\rm p}: (-\D,\D)\to \R$ satisfying 
$$
{\rm p}\in C^{\infty}(-\D,\D), \quad {\rm p}(0)=0, \quad {\rm p}'(0)=0
$$ 
and such that   
$$
\Omega \cap B_{r_1}(\zeta) =\{(x_1,x_2)\,:\, -\D<x_1< \D, \quad {\rm p}(x_1)<x_2\}\cap B_{r_1}(0,0).
$$

Consider the flattening change of variables $F_i:\Omega \cap B_{r_1}(\zeta)\to \R^2$ defined by
$$
F_i:=(F_{i1},F_{i2}), \quad \quad \hbox{where}\quad F_{i1}:=x_1 + \frac{x_2 - {\rm p(x_1)}}{1 + |{\rm p'}(x_1)|^2}{\rm p}'(x_1) \quad \quad F_{i2}:= x_2 - G(x_1).
$$

At this point we remark that the radius $r_{1}$ must satisfy that $r_{1}\in (0,c_{\ep})$, where $c_{\ep}$ is defined in \eqref{distancezetaizetaj}.

\medskip
Throughout our discussions we will assume that $r_{1}\in (\frac{1}{2}c_{\ep},c_{\ep})$, so that $\frac{r_1}{\ep}\to \infty$ as $\ep\to 0$.

\medskip
Recalling that $\zeta_i'=\frac{\zeta_i}{\ep}$, we set for $y\in \pp \Omega_{\ep} \cap B_{\frac{r_1}{\ep}}(\zeta_i')$, 
$$
F^{\ep}_i(y):= \frac{1}{\ep}F_i(\ep y)
$$ 
and define
$$
\chi_{i}(y):= \chi(|F_i^{\ep}(y)|), \quad \quad Z_{ij}(y):=\chi_i(y)z_{ij}(F_i^{\ep}(y)).
$$

The following proposition accounts for the solvability of the linear problem \eqref{lineareqanphi1}-\eqref{boundaryorthogonalityconditionphi1}. For this we define the following norm
\begin{equation}\label{normh}
\|h\|_{*}:=\sup \limits_{y \in \Omega_{\ep}} \frac{|h(y)|}{\ep^2 \,+\,\sum_{i=1}^m \left(1 + |y - \zeta_i'|\right)^{-2-\sigma}}.
\end{equation}

\medskip

\begin{prop}\label{solvabilitylinearproblem}
Assume the conditions of Lemma \ref{sizeof the errorSvep} on the parameters $d_i$ and $\zeta_i$. For any $\ep>0$ small enough and any given $h\in L^{\infty}(\Omega_{\ep})$ there exists $\phi\in C^1(\overline{\Omega_{\ep}})$ and constants $c_{ij}\in \R$ where $\phi$ is the unique solution of \eqref{lineareqanphi1}. Moreover, there exists a constant $C>0$ independent of $\ep>0$ such that
$$
\|\nabla \phi\|_{L^{\infty}(\Omega_{\ep})}\,+\,\|\phi\|_{L^{\infty}(\Omega_{\ep})}\leq C\,\ln\left(\frac{1}{\ep}\right)\|h\|_{*}.
$$
\end{prop}

With our choice of the radius $r_{1}$, the proof of Proposition \ref{solvabilitylinearproblem} follows the same lines of the proof of Proposition 3.1 in \cite{DELPINOWEI} with only slight changes. We leave details to the reader. 

\medskip

Denote by $\zzeta:= (\zeta_1,\ldots,\zeta_m)$ and $T(h)=\phi$ the linear operator given by Proposition \ref{solvabilitylinearproblem}. Clearly $T(h)$ depends on ${\zzeta}$, so we write $$T(h)=T(h)({\zzeta}).
$$ 

Proceeding exactly as in section 3 in \cite{DELPINOKOWALCZYKMUSSO}, we obtain that the mapping $   {\zzeta}\to T(h)(   {\zzeta})$ is differentiable in $   {\zzeta}$ and 
$$
\|D_{   {\zzeta}}T(h)\|\leq C\ln\left(\frac{1}{\ep}\right)^{2}\|h\|_{*}.
$$

\medskip
As a direct application of Proposition \ref{solvabilitylinearproblem} and a fixed point argument, we solve the nonlinear problem

\begin{equation}\label{nonlinearlineareqanphi1}
\Delta \phi -\ep^2 \phi =  -S(v_{\ep})-e^{v_{\ep}}\phi -\ep \gamma(\ep y) \cdot \nabla\phi- N(\phi) +
\sum_{i=1}^m \sum_{j=1,J_i}c_{ij}\chi_{ij} Z_{ij} \quad \hbox{in }\quad \Omega_{\ep}
\end{equation}
with the boundary and orthogonality 
conditions in \eqref{boundaryorthogonalityconditionphi1}.

\begin{prop}\label{solvabilitynonlinearprojectedproblem}
Under assumptions in Propososition \ref{solvabilitylinearproblem}, given any $\A>0$ for every $\ep>0$ small there exists a solution $\phi$ and constants $c_{ij}\in \R^N$ satisfying \eqref{nonlinearlineareqanphi1}-\eqref{boundaryorthogonalityconditionphi1} and such that 
$$
\|\nabla \phi\|_{L^{\infty}(\Omega_{\ep})}\,+\,\|\phi\|_{L^{\infty}(\Omega_{\ep})}\leq C\,\ep^{\A}\ln\left(\frac{1}{\ep}\right).
$$

Even more, $\phi=\Phi({   {\zzeta}})$ is differentiable respect to ${   {\zzeta}}$  and 
$$
\|D_{{   {\zzeta}}}\Phi\|\leq C\ep^{\A}\ln\left(\frac{1}{\ep}\right)^{2}.
$$
\end{prop}

\begin{proof}
We proceed using a fixed point argument, as in the proof of Lemma 4.1 in \cite{DELPINOWEI} using the fact that
$$
\|\nabla \phi\|_*\leq C \|\nabla \phi\|_{L^{\infty}(\Omega_{\ep})}.
$$
\end{proof}

\section{The energy estimate}\label{energy estimates}

  Our next step, consists in choosing the points $\zeta_i\in \overline{\Omega}$ so that the real numbers, (actually functions of $   {\zzeta}:=(\zeta_1,\ldots,\zeta_m)$\,) $c_{ij}$, in equation \eqref{nonlinearlineareqanphi1}, are identically zero. Thus leading to the solution of \eqref{nonlineareqnphi1}-\eqref{boundaryconditionphi1}.

\medskip  
Notice that the dimension of the approximate kernel of the linear problem \eqref{lineareqanphi1}-\eqref{boundaryorthogonalityconditionphi1} is $3k + 2l$. We get rid of $m=k+l$ elements of this approximate kernel by choosing parameters $d_1,\ldots,d_m$ satisfying \eqref{logconditiondi}. On the other hand, associated to the points $\zeta_i \in \Omega$, we must get rid off the constants $c_{i1}$ and $c_{i2}$, while for points $\zeta_i \in \pp \Omega$, we must get rid of only the constants $c_{i1}$.
 \\

Throughout this part, we take $\zeta_1, \ldots,\zeta_m \in \overline{\Omega}$ satisfying conditions \eqref{assumptionszetai} and \eqref{nottoclosezetaij}. In addition, we assume that for some $\kappa \geq 1$
\begin{equation}\label{closebutnotmuch}
c|\ln(\ep)|^{-\kappa}\leq |\zeta_i - \zeta_j| \leq C  , \quad \forall\, i,j =1,\ldots,m, \quad i\neq j.
\end{equation}
Since 
$$
c_{\ep}:=\min\{|\zeta_i -\zeta_j|\,:\, i,j=1,\ldots,m, \quad i\neq j\}
$$
we notice that $c_{\ep}\geq c|\ln(\ep)|^{-\kappa}$.

Since the approximation in \eqref{approx} depends on the points $\zeta_i$, we write
$$
u_{\ep}(   {\zzeta}):=u_{\ep}=\sum_{i=1}^m U_i + H_i, \quad \quad \hbox{in} \quad  \overline{\Omega},
$$
where 
\begin{equation*}U_{i}(x):= \ln
\left(\frac{8d_i^2}{\left(\ep^2d_i^2 + |x-\zeta_i|^2\right)^2}\right),\quad  \quad \hbox{for } x \in \R^2 
\end{equation*}
and the
$d_i$'s are real numbers satisfying
\begin{equation}\label{logconditiondi1}
\ln\left(8 d_i^2\right)= c_iH(\zeta_i,\zeta_i) \,+ \, \sum_{j\neq i} c_jG(\zeta_i,\zeta_j), \quad \quad \forall\,i=1,\ldots,m.
\end{equation}
Denote also 
$
\phi=\phi(   {\zzeta})$ and $c_{ij}(   {\zzeta})
$ 
the unique solution of \eqref{nonlinearlineareqanphi1}-\eqref{boundaryconditionphi1} and finally set $\tilde{\phi}(   {\zzeta}):= \phi(   {\zzeta})(\frac{x}{\ep})$.

\medskip

For any $u\in H^1(\Omega)$, we consider the energy
\begin{equation}\label{energy}
J(u):=\frac{1}{2}\int_{\Omega}a(x)\left(|\nabla u|^2 + u^2\right)dx \,-\,\ep^2\int_{\Omega}a(x)e^u
\end{equation} 
which is well defined due to the Moser-Trudinger inequality. It is well known that critical points of $J$ are nothing but solutions   of equation \eqref{emdenfowlerhighdim}.
Let us introduce the so-called {\emph {reduced energy}}
\begin{equation}\label{red-energy}
\mathcal{F}_{\ep}(   {\zzeta}):= J\left(u_{\ep}(   {\zzeta}) + \tilde{\phi}(   {\zzeta})\right).
\end{equation}

The aim of the next result is to understand the role played by $\mathcal{F}_{\ep}$ and its asymptotic behavior as $\ep\to0.$

\begin{prop}\label{variational reduction}
\begin{itemize}
\item[(i)]
If $   {\zzeta}=(\zeta_1,\ldots,\zeta_m),$ with the $\zeta_i$'s satisfying conditions  \eqref{assumptionszetai}, \eqref{nottoclosezetaij} and \eqref{closebutnotmuch}, is a critical point of the functional $\mathcal{F}_{\ep}$ then $u=u_{\ep}(   {\zzeta}) + \tilde{\phi}(   {\zzeta})$ is a critical point of $J$, i.e. a  solution of equation \eqref{emdenfowlerhighdim}.

\item[(ii)]
There holds true that
$$
\mathcal{F}_{\ep}(   {\zzeta})=\mathcal F_0(\zzeta) + \tilde{\theta}_{\ep}(   {\zzeta})
$$
$C^1-$uniformly in $\zzeta=(\zeta_1,\ldots,\zeta_m)$.
Here (given $c_0:=-4\pi(2-\ln8)$)
\begin{equation}\label{interior}\begin{aligned}
\mathcal F_0(\zzeta):= &\sum_{i=1}^m a(\zeta_i) \left(c_0+8\pi|\ln \ep|\right)  
-4\pi \sum_{i=1}^m
a(\zeta_i)\left[H(\zeta_i,\zeta_i) +\sum_{j\neq i}G(\zeta_i,\zeta_j)\right]\\
&\hbox{if $\zeta_1,\dots,\zeta_m\in\Omega$}\\
\end{aligned}\end{equation}
and
\begin{equation}\label{boundary}\begin{aligned}
\mathcal F_0(\zzeta):= &\sum_{i=1}^m a(\zeta_i) \left(c_0+8\pi|\ln \ep|\right)  
-2\pi \sum_{i=1}^m
a(\zeta_i)\left[H(\zeta_i,\zeta_i) +\sum_{j\neq i}G(\zeta_i,\zeta_j)\right]\\
&\hbox{if $\zeta_1,\dots,\zeta_m\in\pp\Omega$}\\
\end{aligned}\end{equation}
Moreover $\tilde{\theta}_{\ep}$ is a $C^1-$ function such that  
$
|\tilde{\theta}_{\ep}| + |\nabla _{   {\zzeta}} \tilde{\theta}_{\ep}| 
\to 0
$ as $\ep \to 0$,
\end{itemize}
\end{prop}

\begin{proof}
  First of all, we prove $(i)$ and the fact that 
 there exists a $C^1$ differentiable function $\tilde{\theta}_{\ep}$ such that 
$$
\mathcal{F}_{\ep}(   {\zzeta})=J(u_{\ep}(   {\zzeta})) + \tilde{\theta}_{\ep}(   {\zzeta})
$$
and as $\ep \to 0$,
$$
|\tilde{\theta}_{\ep}| + |\nabla _{   {\zzeta}} \tilde{\theta}_{\ep}| 
\to 0
$$
uniformly in $\zzeta=(\zeta_1,\ldots,\zeta_m)$.
 The proof follows the same lines of Lemma 5.1  and Lemma 5.2 in \cite{DELPINOWEI}, with no changes. We leave details to the reader. We also refer the reader to \cite{WEIYEZHOU1,WEIYEZHOU2} for similar computations. 
 \\\\
  Then we analyze the asymptotic behavior of the energy $J(u_{\ep}(\zzeta))$. 
 We  only consider the case when the points belong to $\Omega$, because the proof   when the points lie on the boundary $\pp\Omega$ relies on similar arguments.  {Moreover, we only compute the $C^0-$expansion because the $C^1-$estimate can be obtained in a similar way.}

First of all,   since for every $i=1,\ldots,m$
$$
\rdiv(a(x)\nabla PU_i)-a(x)PU_i+\ep^2a(x)e^{U_i} =0\quad \quad \hbox{in} \quad \Omega
$$
$$
\frac{\pp PU_i}{\pp \nu}=0\quad \quad \hbox{on} \quad \pp \Omega
$$
we know that
for any $j,i=1,\ldots,m$, 
$$
\int_{\Omega}a(x)\left(\nabla PU_i\cdot \nabla PU_j + PU_i\cdot PU_j\right)dx = \ep^2\int_{\Omega}a(x)e^{U_i}PU_jdx.
$$

Hence, we compute
\begin{eqnarray*}
J(u_{\ep}(\zzeta))&=&\frac{1}{2}\int_{\Omega}a(x)\left(|\sum_{i=1}^m \nabla PU_i|^2 + |\sum_{i=1}^m PU_i|^2\right)- \ep^2\int_{\Omega}a(x)e^{\sum_{i=1}^m PU_i}\\
&=&\frac{1}{2}\sum_{i=1}^m \int_{\Omega}a(x)\left(\left|\nabla PU_i\right|^2 + \left|PU_i\right|^2\right)\\
&&+ \frac{1}{2}\sum_{i,j=1 i\neq j}^m \int_{\Omega}a(x)\left(\nabla PU_i\cdot \nabla PU_j + PU_i\cdot PU_j\right)-\ep^2\int_{\Omega}a(x)e^{\sum_{i=1}^m PU_i}
\\
&=&\underbrace{\sum_{i=1}^m\frac{\ep^2}{2}\int_{\Omega}a(x)e^{U_i}PU_i dx}_{I} + \underbrace{\sum_{i,j \,\,i\neq j}\frac{\ep^2}{2}\int_{\Omega}a(x)e^{U_i}PU_jdx}_{II} -\underbrace{\ep^2\int_{\Omega}a(x)e^{\sum_{i=1}^m PU_i}dx}_{III}.
\end{eqnarray*}

\medskip

We first compute $I$. Fix $i=1,\ldots,m$ so that 
$$
\ep^2\int_{\Omega}a(x)e^{U_i}PU_idx=\int_{\Omega}a(x)\frac{8\ep^2d_i^2}{(\ep^2 d_i^2 + |x-\zeta_i|^2)^2}\left(U_i(x) + H_i(x)\right)dx.
$$

From Lemma \ref{expansionHdzeta}, for any $\A\in (0,1)$
$$
U_i(x)+H_i(x)=\ln\left(\frac{1}{(\ep^2 d_i^2 + |x-\zeta_i|^2)^2}\right) + H(x,\zeta_i) + \OO(\ep^{\A}), \quad \quad x\in\Omega.
$$

Using the change of variables $x=\zeta_i +\ep d_i y$, we obtain 
$$
\ep^2\int_{\Omega}a(x)e^{U_i}PU_idx
$$

$$
=\int_{\Omega_{\ep d_i} -\frac{\zeta_i}{\ep d_i}}a(\zeta_i + \ep d_i y)\frac{8}{(1+|y|^2)^2}\left(-\ln(\ep^4d_i^4)+\ln\left(\frac{1}{(1+|y|^2)^2}\right)+ H(\zeta_i + \ep d_i y,\zeta_i) + \OO(\ep^{\A})\right)dy
$$

$$
= \int_{\Omega_{\ep d_i} -\frac{\zeta_i}{\ep d_i}}a(\zeta_i)\frac{8}{(1+|y|^2)^2}\left(-\ln(\ep^4d_i^4)\,+\,\ln\left(\frac{1}{(1+|y|^2)^2}\right)+ H(\zeta_i,\zeta_i) + \OO(\ep^{\A}d_i^{\A}|y|^{\A})\right)dy
$$

\begin{equation}
=\int_{\R^{2}}a(\zeta_i)\frac{8}{(1+|y|^2)^2}\left(-\ln(\ep^4d_i^4)\,+\,\ln\left(\frac{1}{(1+|y|^2)^2}\right)+ H(\zeta_i,\zeta_i) \right)dy + \OO(\ep^{\A}d_i^{\A}).
\end{equation}

Since 
$$
\int_{\R^2}\frac{8}{(1+|y|^2)^2}dy=8\pi \quad  \quad \hbox{and} \quad \quad \int_{\R^2}\frac{8}{(1+|y|^2)^2}\ln\left(\frac{1}{(1+|y|^2)^2}\right)dy=-16\pi,
$$
we find that
\begin{equation}\label{estimateI}
I=-8\pi\sum_{i=1}^m a(\zeta_i)\left(\ln(\ep^2 d_i^2)+1\right) + 4\pi\sum_{i=1}^ma(\zeta_i)H(\zeta_i,\zeta_i) + \OO(\ep^{\A}).
\end{equation}
Here we also used the fact that the parameters $d_i$'s satisfy
$ 
\frac{1}{C}\leq d_i \leq |\ln\left({\ep}\right)|^C
$ 
and 
$ 
\lim_{\ep\to 0^+} \ep d_i =0
$ 
(as it follows directly from \eqref{regularpartGreen} and Proposition \eqref{asymptoticsof H} and  expression and \eqref{logconditiondi}).

Next, we compute $II$. First we notice for $i\neq j$ that 
$$
\ep^2\int_{\Omega}a(x)e^{U_i}PU_jdx
$$
\begin{equation*}
=\int_{\Omega}a(x)\frac{8 \ep^2d_i^2}{(\ep^2d_i^2+|x-\zeta_i|^2)^2}\left(\ln\left(\frac{1}{(\ep^2d_j^2+|x-\zeta_j|^2)^2}\right)
+ H(x,\zeta_j) + \OO(\ep^{\A}d_j^{\A})\right)dy.
\end{equation*}

Recall that 
$$
c_{\ep}:=\min\{|\zeta_i -\zeta_j|\,:\, i,j=1,\ldots,m, \quad i\neq j\}.
$$ 

On the other hand, for $|x -\zeta_j|\geq c_{\ep}$, we have that 
\begin{eqnarray*}
\ln\left(\frac{1}{(\ep^2 d^2_j + |x-\zeta_j|^2)^2}\right)&=&-4\ln\left(|x-\zeta_j|\right)+ \OO_{L^{\infty}(B_{c_{\ep}}(\zeta_i))}(\ep^2d_j^2 |y-\zeta'_j|^{-2})\\
&=&-4\ln(|\ep y -\zeta_j|) + \OO_{L^{\infty}(B_{\frac{c_{\ep}}{\ep}}(\zeta_i'))}(\ep^2 d_j^2 |y-\zeta'_j|^{-2})\\
&=&-4\ln(|\zeta_i -\zeta_j|) + \OO_{L^{\infty}(B_{\frac{c_{\ep}}{\ep}}(\zeta_i'))}\left(\ep^{\A}\right).
\end{eqnarray*}

Also 
$$
H(x,\zeta_j)=H(\zeta_i,\zeta_j)+ \OO_{L^{\infty}(\Omega)}(|x - \zeta_j|^{\A}), \quad \quad \forall\, x\in \Omega.
$$
so that 
\begin{eqnarray*}
\ep^2\int_{\Omega}a(x)e^{U_i}PU_jdx&=&\int_{\Omega}a(\zeta_i)\frac{8}{(1+|y|^2)^2}\left(-4\ln\left(|\zeta_i-\zeta_j|\right)
+ H(\zeta_i,\zeta_j) + \OO(\ep^{\A})\right)dy\\
&=&8\pi a(\zeta_i)G(\zeta_i,\zeta_j) + \OO(\ep^{\A}).
\end{eqnarray*}

Therefore,
\begin{equation}\label{estimateII}
II=\sum_{i,j,\,\,i\neq j}^m 4\pi \,a(\zeta_i)G(\zeta_i,\zeta_j) + \OO(\ep^{\A}).
\end{equation}

Finally, we compute $III$. To do this we appeal to Lemma \ref{sizeof the errorSvep} to obtain that for fix $i=1,\ldots,m$ and for any $x\in B_{c_{\ep}}(\zeta_i)$ 
\begin{eqnarray}
III&=&\ep^2\int_{\Omega}a(x)e^{\sum_{j=1}^m  U_j(x) +H_j(x)}dx \nonumber\\
&=&\ep^2 \sum_{i=1}^m\int_{\Omega}a(x)e^{\sum_{j=1}^m U_j(x) +H_j(x)}\chi_{B_{c_{\ep}}(\zeta_i)}dx + \ep^2\int_{\Omega}a(x)e^{\sum_{j=1}^m U_j(x) +H_j(x)}\chi_{\Omega-\cup_{i=1}^m B_{c_{\ep}}(\zeta_i)}dx\nonumber\\
&=&8\pi \sum_{i=1}^m a(\zeta_i) + \OO(\ep^{\A}).\label{estimateIII}
\end{eqnarray}

Finally putting together estimates \eqref{estimateI}, \eqref{estimateII} and  \eqref{estimateIII} we obtain that 
$$
J(u_{\ep}(\zzeta)) = -8\pi\sum_{i=1}^m a(\zeta_i)\left(2+\ln(\ep^2 d_i^2)\right) +4\pi \sum_{i=1}^m
a(\zeta_i)\left[H(\zeta_i,\zeta_i) +\sum_{j\neq i}G(\zeta_i,\zeta_j)\right] + \OO(\ep^{\A})
$$
and using condition \eqref{logconditiondi1} the desired estimate  follows.
 \end{proof}

\section{Proofs of Theorems}\label{proofoftheorems}

\begin{proof}[Proof of Theorem \ref{theo2}.] 
Let
 $
 {\bfs}:=( {s}_1,\ldots, {s}_m)\in (\pp \Omega)^m $ and $ \bft:=(t_1,\ldots,t_m)\in (\R_+)^m .$  
Let us consider the   \emph{configuration space} 
$$
\Lambda_{\D}:=\{
(\bfs,\bft)\in (\pp \Omega )^m\times (\R_+)^m\,:\, |s_i-s_j|\ge \delta, \ \delta<t_i<1/\delta\}, 
$$
for some $\delta>0$ small,   independent of $\ep>0.$

For any point in $(\bfs ,\bft )\in \Lambda_{\D}$, we set
\begin{equation}\label{zetai}
\zeta_i:=s _{i} + |\ln(\ep)|^{-1}\,t _i\,\nu( { s_{i}}), \quad \quad \forall i=1,\ldots,m.
\end{equation}

Observe that \eqref{assumptionsdi}, \eqref{assumptionszetai}, \eqref{logconditiondi} and \eqref{nottoclosezetaij} hold true.
By Proposition \ref{variational reduction} we have to find a critical point $\zzeta^\ep:=(\zeta_1^\ep,\dots,\zeta_m^\ep)\in \Omega\times\dots\times \Omega$ of the reduced energy $\mathcal F_\ep $ where each  $\zeta_i^\ep$ is as in \eqref{zetai}. If we use the parametrization of the points $\zeta_i$ given in \eqref{zetai}, we are lead to find a critical point $(\bfs^\ep,\bft^\ep)$ of the reduced energy $\mathcal F_\ep(\bfs,\bft).$

By \eqref{interior} and the property of Robin's function we deduce that $\mathcal F_\ep$ reduces to 
 \begin{equation}\label{cruciale} \mathcal F_\ep(\bfs,\bft)= 8\pi|\ln\ep|\left[\sum_{i=1}^m a(s_i)+\Upsilon_\ep(\bfs)\right] + 16\pi\left[\sum_{i=1}^m a(s_i) \ln t_i+t\partial_\nu a(s_i)\right]+\Theta_\ep(\bfs,\bft),  
\end{equation}
 where the smooth functions $\Upsilon_\ep(\bfs)$ only depends on $\bfs $ while $ \Theta_\ep(\bfs,\bft)$ depends on both $\bfs$ and $\bft$ and they are higher order terms, namely
  $|\Upsilon_\ep|,|\nabla  \Upsilon_\ep|,|\Theta_\ep|,|\nabla  \Theta_\ep|\to0$ as $\ep\to0.$ The proof of this claim is postponed to the end.
 
\medskip 
Once the estimate \eqref{cruciale} is proved, the claim easily follows by degree theory. Indeed, using \eqref{cruciale}, let us introduce the continuous functions
  \begin{align}\label{cru3} 
&  {\nabla_{s_i}\mathcal F_\ep(\bfs,\bft)\over |\ln\ep|}= \underbrace{8\pi \nabla  a(s_i )}_{:=\mathcal S_i(s_i)}+o(1),\ i=1,\dots,m\\
 \label{cru4} &  {\nabla_{t_i}\mathcal F_\ep(\bfs,\bft) }= \underbrace{16\pi \left[{a(s_i )\over t_i}+\partial_\nu a(s_i) \right]}_{:=\mathcal T_i(s_i,t_i)}+o(1),\ i=1,\dots,m .
 \end{align}
Under assumption {\it (A1)}, we know that $a$ has $m$ different strict local minima or local maxima points $\zeta^*_1,\dots,\zeta_m^*\in \pp \Omega.$ Set $s_i^*:=\zeta_i^*.$ Therefore, for some $\rho>0$ small enough, the \emph{Brouwer degree} $ \mathrm {deg}(\nabla a, B(s^*_i,\rho),0)$ is well defined and (see, for example, Corollary 2 in \cite{AMANN})
 \begin{equation}\label{cru5}
  \mathrm {deg}(\mathcal S_i, B(s^*_i,\rho),0)=\pm1\not=0\ \hbox{for any}\ i=1,\dots,m.
 \end{equation}

On the other hand, since $\partial_\nu a(\zeta^*_i)<0$, we can choose $\rho$ small enough so that for any $s_i\in B(s^*_i,\rho)$ there exists a unique
   $t_i=t_i(s_i)=-{a(s_i)\over\partial_\nu a(s_i)}$ such that
    $\mathcal T_i(s_i,t_i)=0.$  
  
   Set $t_i^*:=t_i(s^*_i)$ and let $\rho$ smaller if necessary, so that the Brouwer degree  
   $$\mathrm {deg}(   (\mathcal S_i, \mathcal T_i),
    B(s^*_i,\rho)\times B(t_i^*,\rho) ,0) 
    $$ 
    is well defined. Since $\partial_{t_i}\mathcal T_i(s_i,t_i(s_i))<0,$ by \eqref{cru5} (see, for example, Lemma 5.7 in \cite{mopi}) we immediately get that  
    \begin{equation}\label{cru6}
  \mathrm {deg}(   (\mathcal S_i, \mathcal T_i),
    B(s_i^*,\rho)\times B(t_i^*,\rho) ,0)\not=0\ \hbox{for any}\ i=1,\dots,m.
 \end{equation}
 Finally, by \eqref{cru6} using the properties of Brouwer degree, we get
    \begin{equation}\label{cru7}\begin{aligned}
 & \mathrm {deg}(   (\mathcal S_1, \mathcal T_1),\dots,(\mathcal S_m, \mathcal T_m) , 
    \left(B(s_1^*,\rho)\times B(t_1^*,\rho)\right)\times\dots\times \left(B(s_m^*,\rho)\times B(t_m^*,\rho)\right),0)\\
    &=\mathrm {deg}(   (\mathcal S_1, \mathcal T_1),
    B(s_1^*,\rho)\times B(t_1^*,\rho) ,0)\times\dots\times\mathrm {deg}(   (\mathcal S_m, \mathcal T_m),
    B(s_m^*,\rho)\times B(t_m^*,\rho) ,0)\\
    &\not=0.
 \end{aligned}\end{equation}
 Combining \eqref{cru7} with \eqref{cru3} and \eqref{cru4} we deduce that if $\ep$ is small enough there exists  $(\bfs^\ep,\bft^\ep)$ such that
 $\nabla \mathcal F_\ep (\bfs^\ep,\bft^\ep)=0$. In particular, $\bfs^\ep=(s_1^\ep,\dots,s_m^\ep)\to (\zeta^*_1,\dots, \zeta^*_m)$ 
 as $\ep\to0.$
 That concludes the proof.
  \\
  
 Let us prove \eqref{cruciale}.
 
Using a smooth extension of the function $a(x)$ we can perform a Taylor expansion around every point $s_i$ to obtain that 
\begin{equation}\label{taylora(x)}
a(\zeta_i)=a(s _i) + |\ln(\ep)|^{-1}\,t _i \pp_{\nu}a(s _i) + \OO([|\ln(\ep)|^{-1}t _i]^2).
\end{equation}
On the other hand, from Corollary \ref{corolario} and the regularity in \eqref{integrabilityR} of the function $R(z)$ described in \eqref{ellipticequationR20}, we find that for any $\A\in (0,1)$ and any $\zeta_i$ as in \eqref{zetai},
   {\begin{equation}\label{expasionrobin}
H(\zeta_i,\zeta_i)=4\ln\left(\ln\left(\frac{1}{\ep}\right)\right)+4\ln\left(\frac{1}{2\,t_i}\right)
\,+\,{\rm z}(s_i)\,+\,\OO\left([|\ln(\ep)|^{-1}t_i]^{\A}\right).
\end{equation}}

On the other hand, since $R\in C^{\infty}(\R^2-\{0\})$, we can use expression \eqref{regularpartGreen}, Proposition \ref{asymptoticsof H} and the fact that $|s_i-s_j|\ge c>0$, to obtain the expansion 
\begin{equation}\label{boundedgreen}\begin{aligned}
G(\zeta_i,\zeta_j)&= - 4\ln\left(|s_i-s_j|\right)\,+\,\tilde{z}(s_i,s_j)
\\ &
+\frac{1}{\ln\left(\frac{1}{\ep}\right)}\left[4 t_j\,R(s_i-s_j)\cdot \left(D\gamma(s_j)\cdot \nu(s_j)\right)+ \nabla \tilde{z}(s_i,s_j)\cdot (t_i\nu(s_i),t_j\nu(s_j))\right]\\
&+  o\left(|\ln(\ep)|^{-1}[|t_i|+ |t_j|]\right).
\end{aligned}
\end{equation} 

Hence, from \eqref{interior}, putting together expressions \eqref{taylora(x)}, \eqref{expasionrobin} and \eqref{boundedgreen}, we get
$$\begin{aligned}
\mathcal F_0(\bfs,\bft)&=
 \sum_{i=1}^m\left[{\rm c}_0 + {\rm c}_1\ln\left(\ln\left(\frac{1}{\ep}\right)\right) +{\rm c}_2\ln\left(\frac{1}{\ep}\right)\right]a(s_i)\\ &
+16\pi\sum_{i=1}^m\left[\ln\left(t_i\right)a(s_i) + \left(1 + \frac{\ln\left(t_i\right)}{\ln\left(\frac{1}{\ep}\right)}\right) t_i\,\pp_{\nu}a(s_i)\right]\\ &
-4\pi\sum_{i=1}^m \left[ {\rm z}(s_i)  - \sum_{j\neq i} 4\ln\left(|s_i -s_j|\right) + \tilde{z}(s_i,s_j)\right]\left(a(s_i)\,+\,\frac{t_i}{\ln\left(\frac{1}{\ep}\right)}\pp_{\nu}a(s_i)\right)
\\ &+\OO\left(|\ln(\ep)|^{-\A}|\bft|^{\A}\right)\end{aligned}
$$
for some constants ${\rm c}_0,{\rm c}_1$ and $,{\rm c}_2>0$ and estimate \eqref{cruciale} follows.\\

 This concludes the proof of the Theorem.
\end{proof}
 
 \bigskip 
We next proceed with the proof of Theorem \ref{theo3}.

\bigskip
\begin{proof}[Proof of Theorem \ref{theo3}.]  
  By Proposition \ref{variational reduction} we have to find a critical point $\zzeta^\ep:=(\zeta_1^\ep,\dots,\zeta_m^\ep)\in\partial\Omega\times\dots\times\partial\Omega$ of the reduced energy $\mathcal F_\ep.$ By \eqref{boundary} we have that $\mathcal F_\ep$ reduces to 
  $$\mathcal F_\ep(\zzeta)=|\ln\ep|\left[8\pi\sum\limits_{i=1}^m a(\zeta_i)+o(1)\right]$$
  $C^1-$uniformly in $ \left\{(\zeta_1,\dots,\zeta_m)\in\pp\Omega\times\dots\times\pp\Omega\ :\  |\zeta_i-\zeta_j|\ge c\ \hbox{for any }\ i\not=j\right\}$. The claim follows from the fact that $a$ has $m$ different strict local maxima or local minima points  $\zeta^*_1,\dots,\zeta^*_m$ on the boundary $\pp\Omega$ which are stable under $C^1-$perturbation, using a degree argument as in the proof of Theorem \ref{theo2}. In particular, there exists a critical point $(\zeta_1^\ep,\dots,\zeta_m^\ep)$ of $\mathcal F_\ep$ such that each $\zeta_i^\ep\to \zeta^*_i$ for $i=1,\dots,m.$
\end{proof}

  \bigskip
We next proceed with the proof of Theorem \ref{theo4}.

\bigskip
\begin{proof}[Proof of Theorem \ref{theo4}.]  By Proposition \ref{variational reduction} we have to find a critical point $\zzeta^\ep:=(\zeta_1^\ep,\dots,\zeta_m^\ep)\in\partial\Omega\times\dots\times\partial\Omega$ of the reduced energy $\mathcal F_\ep.$ By \eqref{boundary} we have that $\mathcal F_\ep$ reduces to 
  $$\mathcal F_\ep(\zzeta)= 8\pi|\ln\ep| \sum\limits_{i=1}^m a(\zeta_i)-2\pi\sum_{j,i=1\atop j\neq i}^m a(\zeta_i) G(\zeta_i,\zeta_j)+o(1) $$
 $C^0-$uniformly in $ \left\{(\zeta_1,\dots,\zeta_m)\in\pp\Omega\times\dots\times\pp\Omega\ :\ |\zeta_i-\zeta_j|\ge c|\ln\ep|^{-\kappa}\ \hbox{for any }\ i\not=j\right\}$.

Arguing exactly as in Lemma 7.1 of \cite{WEIYEZHOU2} we prove that $\mathcal F_\ep$ has a local maximum point $(\zeta_1^\ep,\dots,\zeta_m^\ep)$ such that each $\zeta_i^\ep\to \zeta_0$ as $\ep\to0.$  
\end{proof}

\section{Appendix}
The following technical lemma is rather standard, but for the sake of completeness, we include the details of the proof.

\begin{lemma}\label{boundaryterm}
For every $x, \zeta \in \pp\Omega$, define
$g(x,\zeta)$ as
\begin{equation*}
g(x,\zeta):=
\left\{
\begin{array}{ccc}
\frac{\nu(x)\cdot(x-\zeta)}{|x -\zeta|^2},& \quad x \in \pp \Omega-\{\zeta\},\\
\\
\frac{1}{2}k_{\pp \Omega}(\zeta), & \quad x = \zeta,
\end{array}
\right.
\end{equation*}
where $k_{\pp \Omega}(y)$ is the signed curvature of $\pp \Omega$ at $y$. Then for every $\zeta \in 
\pp \Omega$, $g(\cdot,\zeta)\in  C^{\infty}(\pp \Omega)$. Even more, the mapping $\zeta \in \pp \Omega \mapsto g(\cdot, \zeta)$ 
belongs to $C^1(\pp \Omega; C^1(\pp \Omega))$. 
\end{lemma}

\begin{proof} Recall that $\pp \Omega$ is smooth. Let $\zeta \in \pp \Omega$ be arbitrary. After a translation and a rotation, if necessary, we may assume that $\zeta=(0,0)\in \R^2$ and that $\nu(\zeta)=(0,1)$. Hence, there exists $R>0$, $\D>0$ small and a function ${\rm p}: (-\D,\D)\to \R$ satisfying 
$$
{\rm p}\in C^{\infty}(-\D,\D), \quad {\rm p}(0)=0, \quad {\rm p}'(0)=0
$$ 
and such that   
$$
\Omega \cap B_R(\zeta) =\{(x_1,x_2)\,:\, -\D<x_1< \D, \quad {\rm p}(x_1)<x_2\}\cap B_{R}(0,0).
$$

\medskip
The inner unit normal vector in $\pp \Omega \cap B_R(\zeta)$ is computed as
$$
\nu(x_1,{\rm p}(x_1))=\frac{(-{\rm p}'(x_1),1)}{\sqrt{1 +|{\rm p}'(x_1)|^2 }}, \quad x_1 \in (-\D,\D),
$$
so that  for $x_1\in (-\D,\D)$
\begin{equation}\label{smoothnessboundarycondition}
g(x_1,{\rm p}(x_1),\zeta)=\frac{1}{\sqrt{1 +|{\rm p}'(x_1)|^2 }}\cdot
\frac{{\rm p}(x_1) - x_1\,{\rm p}'(x_1)}{x_1^2 + {\rm p}^2(x_1)}.
\end{equation}

Since ${\rm p}$ is smooth, we use Taylor's expansion together with  the fact that ${\rm p}(0)={\rm p}'(0)=0$ to find that
\begin{equation}\label{redefinitionbdcondregularpart}
{\rm p}(x_1) = x_1^2\int_{0}^1(1-t) {\rm p}''(t\, x_1)dt\,=\,x_1^2 \,{\rm r}(x_1), \quad \quad {\rm p}'(x_1)=x_1\int_0^1 {\rm p}''(t\,x_1)dt\,=\, x_1\,{\rm q}(x_1),
\end{equation}
where
$$
{\rm r}(x_1):=\int_{0}^1(1-t) {\rm p}''(t\, x_1)dt, \quad \quad {\rm q}(x_1):=\int_0^1 {\rm p}''(t\,x_1)dt, \quad \quad  \hbox{for } \quad -\D<x_1<\D.
$$

\medskip 
Putting together \eqref{smoothnessboundarycondition} and \eqref{redefinitionbdcondregularpart} 
we obtain
\begin{equation}\label{eqrmpsmooth}
\frac{{\rm p}(x_1) - x_1\,{\rm p}'(x_1)}{x_1^2 + {\rm p}(x_1)^2}=\frac{{\rm r}(x_1) - {\rm q}(x_1)}{1 + x_1^2\,{\rm r}^2(x_1)}
\end{equation}
which is a smooth function in $(-\D,\D)$.

\medskip
By succesive differentiation of the identity in \eqref{eqrmpsmooth}, we conclude that $g(x,\zeta)$ is smooth and also notice that 
$$
\lim_{x_1\to 0}g((x_1,p(x_1)),\zeta)=-\frac{1}{2}{\rm p}''(0)=\frac{1}{2}k_{\pp \Omega}(0,0).
$$ 
Finally from the smoothness of the tangent bundle of $\pp \Omega$, $T(\pp \Omega)$, we obtain that $\zeta \mapsto g(\cdot,\zeta)$ belongs to $C^1(\pp \Omega; C^1(\pp \Omega))$
and this concludes the proof of the lemma. Using
\end{proof}

\medskip

{\bf Acknowledgments: } The research of the first author was supported by the Grant 13-00863S of the Grant Agency of the Czech Republic. 
The research of the second author was partially supported by GNAMPA.

\end{document}